\documentclass[reqno]{amsart}

\usepackage{amsmath,amssymb,amsthm,mathtools}
\usepackage{longtable,booktabs,array}
\usepackage[hidelinks]{hyperref}
\hypersetup{
  pdftitle={Powers of the Thue--Morse Series: 2-Adic Valuations and Automatic Odd Parts},
  pdfauthor={Zhao Shen},
  pdfsubject={2-adic valuations and automatic sequences},
  pdfkeywords={Thue--Morse sequence, automatic sequence, 2-adic valuation,
    Mahler functional equation}
}

\newtheorem{theorem}{Theorem}[section]
\newtheorem{proposition}[theorem]{Proposition}
\newtheorem{lemma}[theorem]{Lemma}

\newtheorem{conjecture}[theorem]{Conjecture}

\newtheorem{remark}[theorem]{Remark}

\newcommand{\repthmname}{}
\newtheorem*{repthm}{\repthmname}
\newenvironment{retheorem}[1]{%
  \renewcommand{\repthmname}{Theorem~\ref{#1}}%
  \begin{samepage}\begin{repthm}%
}{\end{repthm}\end{samepage}}

\begin{document}

\title[Powers of the Thue--Morse Series]{Powers of the Thue--Morse Series: \texorpdfstring{$2$}{2}-Adic Valuations and Automatic Odd Parts}

\author{Zhao Shen}

\address{Department of Mathematics, Central South University}
\email{sz1021@csu.edu.cn}


\begin{abstract}
Let
\[
T(x)=\prod_{j\ge 0}(1-x^{2^j})=\sum_{n\ge 0} t(n)x^n
\]
be the Thue--Morse generating function, and write
$T(x)^m=\sum_{n\geq0}t_m(n)x^n$ for a positive integer $m$.  For a
nonzero integer $a$, let
$\nu_2(a)$ be the exponent of $2$ in $a$ and put
$\operatorname{odd}(a)=a/2^{\nu_2(a)}$.  For every $s\geq1$, we prove
that the sequence
\[
\bigl(\operatorname{odd}(t_m(n))\bmod 2^s\bigr)_{n\geq0}
\]
is $2$-automatic when $m=2^r$, $r\geq1$, or $m=3\cdot2^r$, $r\geq2$.
In both families,
\[
\nu_2(t_m(n))=\nu_2\binom{n+m-1}{m-1}.
\]
The valuation identity for $m=2^r$ is known; our proof recovers it and also
establishes the automaticity assertion.  For $m=6$ we prove
\[
\nu_2(t_6(n))=\nu_2\binom{n+5}{5}+\mathbf{1}_{n\equiv 3\!\!\pmod 4},
\]
where $\mathbf 1_{\mathcal E}$ is $1$ if the statement $\mathcal E$ holds
and is $0$ otherwise, and show that its odd parts modulo every $2^s$ are
$2$-automatic.
\end{abstract}

\maketitle

\section{Introduction}

The Thue--Morse sequence is defined by
\[
t(0)=1, \qquad t(2n)=t(n), \qquad t(2n+1)=-t(n)
\qquad (n\ge 0),
\]
and its generating function is
\[
T(x)=\sum_{n\ge0}t(n)x^n=\prod_{j\ge0}(1-x^{2^j}).
\]
We refer to \cite{AllSh} for background on this sequence.
For each integer $m\ge1$, we write
\begin{equation}\label{eq:001}
T(x)^m=\sum_{n\ge0}t_m(n)x^n.
\end{equation}

For a nonzero integer $a$, the $2$-adic valuation $\nu_2(a)$ is the
largest integer $v$ such that $2^v$ divides $a$; we put
$\nu_2(0)=+\infty$.  The signed odd part of $a$ is
\[
\operatorname{odd}(a)=\frac{a}{2^{\nu_2(a)}}\quad(a\ne0),
\qquad \operatorname{odd}(0)=0.
\]
If $\mathcal E$ is a statement, then $\mathbf1_{\mathcal E}$ is $1$ when
$\mathcal E$ is true and $0$ otherwise.  We write
$\mathbb F_2=\mathbb Z/2\mathbb Z$.

A finite automaton has finitely many states and updates its state as it reads
input digits.  A sequence $(a_n)_{n\ge0}$ with values in a finite set is
$2$-automatic if such an automaton, reading the binary expansion of $n$
from the most significant digit, outputs $a_n$.  Equivalently, its
$2$-kernel
\[
\left\{(a_{2^e n+c})_{n\ge0}:e\ge0,\ 0\le c<2^e\right\}
\]
is finite \cite[Theorem~6.6.2]{AllSh1}.  An integer sequence is
$2$-regular if the $\mathbb Z$-module generated by its $2$-kernel is
finitely generated \cite{AlloucheShallitRegular}.

The functional equation
\begin{equation}\label{eq:002}
T(x)=(1-x)T(x^2)
\end{equation}
gives recurrences for the even- and odd-indexed coefficients of $T(x)^m$.
Gawron, Miska, and Ulas \cite{GawronMiskaUlas} used these recurrences to
study the arithmetic of $t_m(n)$.  In particular, for $m=2^r$ with
$r\ge1$, they proved that
\begin{equation}\label{eq:003}
\nu_2(t_m(n))=\nu_2\binom{n+m-1}{m-1}
\qquad(n\ge0).
\end{equation}
They also showed that the valuation sequence is $2$-regular
\cite[Corollary~3.4]{GawronMiskaUlas}.  Thus \eqref{eq:003} determines the
exact power of $2$ dividing $t_m(n)$ in this family.  Our focus is the odd
factor that remains after this power is removed.

For powers of more general integer-coefficient series, Ulas
\cite{UlasValuations} gave sufficient conditions for bounded $2$-adic
valuations, including applications to the Rudin--Shapiro sequence.
Our question concerns odd parts in families where the valuations are
unbounded.  For each fixed $m$ and $L\ge1$, the sequence
$(t_m(n)\bmod2^L)_{n\ge0}$ is $2$-automatic by the standard closure
properties of regular sequences; see Lemma~\ref{lem:04}.  This does not
settle the corresponding question for odd parts.  When $t_m(n)\ne0$,
determining $\operatorname{odd}(t_m(n))\bmod2^s$ requires the residue of
$t_m(n)$ modulo $2^{\nu_2(t_m(n))+s}$, once the valuation is known.
The required precision can therefore grow with $n$.

Appendix~\ref{app:02} shows that this gap is genuine in a broader setting.
It constructs a $2$-automatic sequence $(a_n)_{n\ge0}$ with values in
$\{-1,1\}$ such that, on writing
\[
\left(\sum_{n\ge0}a_nx^n\right)^4=\sum_{n\ge0}c(n)x^n,
\]
every $c(n)$ is nonzero, while
$(\operatorname{odd}(c(n))\bmod4)_{n\ge0}$ is not $2$-automatic.  Thus
$2$-automaticity of the original coefficient sequence does not, in general,
pass to the odd parts of the coefficients of a power.  The sequence
$(c(n))_{n\ge0}$ is $2$-regular, so $(c(n)\bmod2^L)_{n\ge0}$ is
$2$-automatic for every fixed $L$.

For the Thue--Morse series, we prove that the odd parts are $2$-automatic
modulo every fixed power of $2$ for two infinite families of even exponents
and for $m=6$.  The first theorem includes a new proof of the known
valuation formula \eqref{eq:003}.

\begin{theorem}\label{thm:pow2}
Let $r\ge1$ and put $m=2^r$.  For every $n\ge0$,
\[
\nu_2(t_m(n))=\nu_2\binom{n+m-1}{m-1}.
\]
Moreover, for every $s\ge1$, the sequence
\[
\bigl(\operatorname{odd}(t_m(n))\bmod2^s\bigr)_{n\ge0}
\]
is $2$-automatic.
\end{theorem}

For $m=3\cdot2^r$ with $r\ge2$, we obtain the same valuation formula.

\begin{theorem}\label{thm:val3}
Let $r\ge2$ and put $m=3\cdot2^r$.  For every $n\ge0$,
\[
\nu_2(t_m(n))=\nu_2\binom{n+m-1}{m-1}.
\]
\end{theorem}

The odd parts in this family satisfy the following additional property.

\begin{theorem}\label{thm:odd3}
Let $r\ge2$, put $m=3\cdot2^r$, and let $s\ge1$.  Then the sequence
\[
\bigl(\operatorname{odd}(t_m(n))\bmod2^s\bigr)_{n\ge0}
\]
is $2$-automatic.
\end{theorem}

The remaining case $m=6$ has a correction term in its valuation formula.
Its odd parts are also $2$-automatic modulo every fixed power of $2$.

\begin{theorem}\label{thm:six}
For every $n\ge0$,
\[
\nu_2(t_6(n))=\nu_2\binom{n+5}{5}
+\mathbf1_{n\equiv3\!\!\pmod4}.
\]
Moreover, for every $s\ge1$, the sequence
$\bigl(\operatorname{odd}(t_6(n))\bmod2^s\bigr)_{n\ge0}$ is $2$-automatic.
\end{theorem}

Apart from the valuation identity for $m=2^r$, the valuation results and
all the odd-part automaticity statements in
Theorems~\ref{thm:pow2}--\ref{thm:six} are, to the best of the author's
knowledge, new.  The valuation formulas also show that none of the
coefficients in these families vanishes.

In Section~\ref{sec:07}, we formulate two conjectures.  One predicts
automatic odd parts for every positive exponent $m$.  The other gives a
criterion for the even exponents satisfying \eqref{eq:003} at every
index.  The criterion is expressed in terms of blocks of consecutive
$1$'s in the binary expansion of an integer determined by $m$.
The conjectures are motivated by the proved cases and exploratory
numerical computations; no theorem depends on these experiments.

Section~\ref{sec:01} collects the recurrences and automaticity lemmas.
Section~\ref{sec:03} proves Theorem~\ref{thm:pow2}.
Theorems~\ref{thm:val3} and~\ref{thm:odd3} are proved in
Section~\ref{sec:05}, and Section~\ref{sec:04} proves
Theorem~\ref{thm:six}.

\section{Recurrences and automaticity lemmas}
\label{sec:01}

The functional equation and its $m$-th power are
\[
T(x)=(1-x)T(x^2),
\qquad
T(x)^m=(1-x)^mT(x^2)^m.
\]
Throughout this section, and whenever a recurrence reaches a negative
index, we use the convention $t_m(a)=0$ for $a<0$.  Comparing even and odd
coefficients gives the recurrences of
\cite[Lemma~3.1]{GawronMiskaUlas}:
\begin{equation}\label{eq:004}
t_m(2n)=\sum_{j=0}^{\lfloor m/2\rfloor}
\binom{m}{2j}t_m(n-j)
\end{equation}
and
\begin{equation}\label{eq:005}
t_m(2n+1)=-\sum_{j=0}^{\lfloor(m-1)/2\rfloor}
\binom{m}{2j+1}t_m(n-j).
\end{equation}

We also record the binary digit-sum form of the binomial valuation.  Let
$s_2(a)$ denote the number of $1$'s in the binary expansion of $a$.  The
following standard identity is the binary case of Legendre's formula; see
also \cite[proof of Theorem~3.3]{GawronMiskaUlas}.

For $m\geq1$ and $n\geq0$,
\begin{equation}\label{eq:006}
\nu_2\binom{n+m-1}{m-1}
=s_2(n)+s_2(m-1)-s_2(n+m-1).
\end{equation}

For the two infinite families considered below, it is useful to apply all
shifted versions of these recurrences simultaneously.  The following vector
formulation will be used with two different diagonal rescalings.

\begin{lemma}\label{lem:02}
Fix $m\geq2$ and put
\[
\mathcal V_n=
\bigl(t_m(n-1),t_m(n-2),\ldots,t_m(n-m+1)\bigr)^T,
\]
where the superscript $T$ denotes transpose.  Let $A$ be the $(m-1)\times(m-1)$ integer matrix defined by
\begin{equation}\label{eq:007}
A_{ij}=(-1)^i\binom{m}{2j-i}
\qquad(1\leq i,j\leq m-1),
\end{equation}
where the binomial coefficient is zero when its lower index lies outside
$[0,m]$.  Then
\begin{equation}\label{eq:008}
\mathcal V_{2n}=A\mathcal V_n
\qquad(n\geq0).
\end{equation}
Consequently,
\[
\mathcal V_{2^a n}=A^a\mathcal V_n
\qquad(a,n\geq0).
\]
\end{lemma}

\begin{proof}
The $i$-th coordinate on the left is $t_m(2n-i)$.  In the coefficient
identity coming from
$T(x)^m=(1-x)^mT(x^2)^m$, a term $t_m(n-j)$ can occur only with the
power $2j-i$ from $(1-x)^m$.  Its coefficient is
\[
(-1)^{2j-i}\binom{m}{2j-i}
=(-1)^i\binom{m}{2j-i}.
\]
The nonzero terms satisfy $0\leq2j-i\leq m$.  Since
$1\leq i\leq m-1$, these inequalities imply $1\leq j\leq m-1$,
so no coefficient outside the chosen window is needed.  This proves
\eqref{eq:008}; iteration proves the last formula.
\end{proof}

The next observation is the algebraic identity used in both infinite-family
arguments.  The family-dependent work will be to prove that the rescaled
matrix is integral and to determine it modulo $2$.

\begin{lemma}\label{lem:03}
Retain $A$ and $\mathcal V_n$ from Lemma~\ref{lem:02}.
Let $D$ be an invertible diagonal matrix over $\mathbb Q$.  Suppose that
there are an integer matrix $M$ and positive integers $e,\ell$ such that
\[
2^eMD=DA^\ell.
\]
Then, for all $a,n\geq0$,
\[
2^{ea}M^aD\mathcal V_n=D\mathcal V_{2^{\ell a}n}.
\]
\end{lemma}

\begin{proof}
Iteration gives $2^{ea}M^aD=DA^{\ell a}$, and hence
\[
2^{ea}M^aD\mathcal V_n
=DA^{\ell a}\mathcal V_n
=D\mathcal V_{2^{\ell a}n}
\]
by Lemma~\ref{lem:02}.
\end{proof}

The same diagonal weights will be used in all three families treated below.
For $1\leq i<m$, define
\[
w_i=\nu_2\binom{m-2}{i-1},
\qquad
D=\operatorname{diag}(2^{w_1},\ldots,2^{w_{m-1}}).
\]
Their explicit form depends on $m$ and will be determined in the relevant
sections.

For $e\geq0$ and $0\leq c<2^e$, the subsequence
$(z_{2^e n+c})_{n\geq0}$ will be called a binary decimation of $(z_n)$.
The collection of all these decimations is its $2$-kernel.  A subset of the
nonnegative integers is called $2$-automatic when its characteristic
sequence is $2$-automatic.  A sequence is eventually periodic if it is
periodic from some index onward.

The following facts will be used for both infinite families and for
$m=6$.

\begin{lemma}\label{lem:04}
For every fixed $m\geq1$, the sequence $(t_m(n))_{n\geq0}$ is
$2$-regular.  In particular, for every $L\geq1$,
\[
\bigl(t_m(n)\bmod2^L\bigr)_{n\geq0}
\]
is $2$-automatic.
\end{lemma}

\begin{proof}
The Thue--Morse sequence is $2$-automatic, and therefore $2$-regular.
The coefficient identity
\[
t_m(n)=
\sum_{\substack{n_1,\ldots,n_m\geq0\\n_1+\cdots+n_m=n}}
t(n_1)\cdots t(n_m)
\]
expresses $t_m$ as its $m$-fold convolution.  Regular sequences are closed
under convolution \cite[Theorem~3.1]{AlloucheShallitRegular}, so $t_m$ is
$2$-regular.  Reducing an integer-valued regular sequence modulo a fixed
integer gives an automatic sequence
\cite[Corollary~2.4]{AlloucheShallitRegular}.
\end{proof}

\begin{lemma}\label{lem:05}
Let $(z_n)_{n\geq0}$ be a $2$-automatic sequence with values in a finite
set.  Then
\[
\bigl(z_{\operatorname{odd}(n+1)}\bigr)_{n\geq0},
\qquad
\bigl(z_{2\operatorname{odd}(n+1)}\bigr)_{n\geq0}
\]
are $2$-automatic.  Moreover, if $(c_j)_{j\geq0}$ is eventually periodic
with values in a finite set, then
\[
\bigl(c_{\nu_2(n+1)}\bigr)_{n\geq0}
\]
is $2$-automatic.
\end{lemma}

\begin{proof}
The identities
\[
z_{\operatorname{odd}(2n+1)}=z_{2n+1},
\qquad
z_{\operatorname{odd}(2n+2)}
=z_{\operatorname{odd}(n+1)}
\]
show that every binary decimation of the first sequence is either the
sequence itself or an element of the $2$-kernel of $z$.  Its $2$-kernel is
therefore finite.  Replacing $z_n$ by the automatic subsequence $z_{2n}$
proves the second assertion.

For the final assertion, put
\[
H_j(n)=c_{j+\nu_2(n+1)}.
\]
Then
\[
H_j(2n)=c_j,
\qquad
H_j(2n+1)=H_{j+1}(n).
\]
Only finitely many distinct $H_j$ occur because $(c_j)$ is eventually
periodic.  Together with the finitely many constant sequences of value
$c_j$, these sequences are closed under binary decimation.  Thus $H_0$
has finite $2$-kernel, and the kernel characterization
of automatic sequences \cite[Theorem~6.6.2]{AllSh1} completes the proof.
\end{proof}

Fixed shifts, finite direct products, and maps between finite sets
preserve automaticity \cite{AllSh1}; a backward shift can be completed
with any fixed finite initial segment.  In particular, a vector formed from
finitely many fixed shifts of $t_m(n)$ is automatic modulo $2^L$ by
Lemma~\ref{lem:04}.

We shall also use two elementary ways of joining finitely many automatic
sequences.  Both are immediate from the kernel characterization of
automatic sequences \cite[Theorem~6.6.2]{AllSh1} and closure under finite
direct products.

\begin{lemma}\label{lem:06}
Let all sequences below take values in finite sets.
\begin{enumerate}
\item Fix $d\geq0$.  If each $(f_c(n))_{n\geq0}$ is $2$-automatic for
$0\leq c<2^d$, then the sequence defined by
\[
f(2^dn+c)=f_c(n)
\qquad(0\leq c<2^d)
\]
is $2$-automatic.
\item Let $E_1,\ldots,E_v$ be $2$-automatic sets forming a partition of
$\mathbb Z_{\geq0}$, and let $g_1,\ldots,g_v$ be $2$-automatic sequences.
Then the sequence which equals $g_i(n)$ for $n\in E_i$ is $2$-automatic.
\end{enumerate}
\end{lemma}

\section{The family \texorpdfstring{$m=2^r$}{m equal to a power of 2}}
\label{sec:03}

We begin with the simplest case.  For $m=2$, the scalar recurrences give
\[
t_2(2n)=t_2(n)+t_2(n-1),\qquad
t_2(2n+1)=-2t_2(n).
\]
Also $T(x)^2\equiv(1-x^2)^{-1}\pmod2$, so $t_2(2n)$ is odd.
Writing $n+1=2^\alpha u$ with $u$ odd and iterating the odd-index
recurrence yields
\[
t_2(n)=(-2)^\alpha t_2(u-1),\qquad
\nu_2(t_2(n))=\alpha=\nu_2(n+1).
\]
The matrix argument below extends this extraction of powers of $2$ to
every $m=2^r$.  It recovers the known valuation formula
\cite{GawronMiskaUlas} and proves automaticity of the odd parts.

\begin{retheorem}{thm:pow2}
Let $r\ge1$ and put $m=2^r$.  For every $n\ge0$,
\begin{equation}\label{eq:020}
\nu_2(t_m(n))=\nu_2\binom{n+m-1}{m-1}.
\end{equation}
Moreover, for every $s\ge1$, the sequence
\[
\bigl(\operatorname{odd}(t_m(n))\bmod2^s\bigr)_{n\ge0}
\]
is $2$-automatic.
\end{retheorem}

Throughout the proof, put
\[
m=2^r,\qquad d=\frac m2.
\]
We retain the coefficient vector $\mathcal V_n$ and the one-step matrix $A$
from Lemma~\ref{lem:02}.

\subsection{The rescaled one-step matrix}

The common weights defined in Section~\ref{sec:01} simplify in this family.  We use the
standard row vectors
\[
e_i=(\underbrace{0,\ldots,0}_{i-1},1,
     \underbrace{0,\ldots,0}_{m-1-i})
\qquad(1\leq i<m).
\]
Thus $e_iB$ is the $i$-th row of a matrix $B$.

\begin{lemma}\label{lem:08}
For $1\leq i<m$,
\begin{equation}\label{eq:021}
w_i=\nu_2(i).
\end{equation}
Define
\begin{equation*}
N=\frac12DAD^{-1}.
\end{equation*}
Then $N$ is an integer matrix.  If $B=N\bmod2$, then
\begin{equation}\label{eq:022}
e_iB=
\begin{cases}
e_d,&i\ \text{odd},\\
e_{i/2}+e_d+e_{d+i/2},&i\ \text{even}.
\end{cases}
\end{equation}
\end{lemma}

The diagonal matrix acts on the coefficients as follows:
\[
D\mathcal V_n=
\begin{pmatrix}
2^{w_1}t_m(n-1)\\
2^{w_2}t_m(n-2)\\
\vdots\\
2^{w_{m-1}}t_m(n-m+1)
\end{pmatrix}.
\]
The factor $2^{-w_h}$ in a recovered coefficient will therefore come
from the $h$-th diagonal entry.

\begin{proof}
For $0<\ell<m$, the identity
\[
\binom m\ell=\frac m\ell\binom{m-1}{\ell-1}
\]
and Lucas' theorem give
\begin{equation}\label{eq:023}
\nu_2\binom m\ell=r-\nu_2(\ell),
\end{equation}
because every binomial coefficient in row $m-1=2^r-1$ is odd.
Also $\nu_2(m-i)=\nu_2(i)$ for $0<i<m$.  Hence
\[
\binom{m-2}{i-1}
=\frac{i(m-i)}{m(m-1)}\binom mi
\]
has valuation $\nu_2(i)$, proving \eqref{eq:021}.

By \eqref{eq:007}, the $(i,j)$ entry of $N$ is
\begin{equation*}
N_{i,j}
=(-1)^i2^{w_i-w_j-1}\binom m{2j-i}.
\end{equation*}
We determine its valuation whenever the binomial coefficient is nonzero.

Suppose first that $i$ is odd.  Then $w_i=0$ and $2j-i$ is odd, so
\eqref{eq:023} gives
\[
\nu_2(N_{i,j})=r-\nu_2(j)-1.
\]
This is nonnegative, and it is zero precisely when $j=d$.  Thus the odd
rows of $B$ equal $e_d$.

Now let $i=2c$, where $1\leq c<d$.  The nonzero entries occur for
$c\leq j\leq c+d$.  At the two endpoints $j=c$ and $j=c+d$, the lower
index of the binomial coefficient is respectively $0$ and $m$; both entries
of $N$ are odd.  For an interior entry write $j=c+z$, where $1\leq z<d$.
Using $w_j=\nu_2(j)$ and \eqref{eq:023} gives
\begin{equation*}
\nu_2(N_{2c,c+z})
=r-1+\nu_2(c)-\nu_2(z)-\nu_2(c+z).
\end{equation*}
If $\nu_2(c)<\nu_2(z)$, the right-hand side is $r-1-\nu_2(z)>0$.
If $\nu_2(c)>\nu_2(z)$, it equals
$r-1+\nu_2(c)-2\nu_2(z)\geq r-\nu_2(z)>0$.  If
$\nu_2(c)=\nu_2(z)$, it equals
\[
r-1-\nu_2(c+z),
\]
which is nonnegative and vanishes precisely when $c+z=d$.  Hence the only
odd entries in row $2c$ occur in columns $c,d,c+d$.  This proves both the
integrality of $N$ and \eqref{eq:022}.
\end{proof}

Iteration of \eqref{eq:022} determines the rows of every power of $B$.
After $r$ steps every entry is $1$.

\begin{lemma}\label{lem:09}
Let
\[
\mathbf J=
\begin{pmatrix}
1&1&\cdots&1\\
1&1&\cdots&1\\
\vdots&\vdots&&\vdots\\
1&1&\cdots&1
\end{pmatrix}_{(m-1)\times(m-1)}.
\]
Then
\begin{equation*}
N^a\equiv\mathbf J\pmod2
\qquad(a\geq r).
\end{equation*}
\end{lemma}

\begin{proof}
It is enough first to compute the row supports of $B=N\bmod2$, by which we
mean the columns containing a $1$ in each row.  If $q$ is
a divisor of $m$, put
\[
U_q=\sum_{\substack{1\leq j<m\\q\mid j}}e_j,
\qquad
V_{q,c}=\sum_{\substack{1\leq j<m\\j\equiv c\;(\mathrm{mod}\ q)}}e_j.
\]
We claim that, for $1\leq a\leq r$ and $q_a=m/2^a$,
\begin{equation}\label{eq:024}
e_iB^a
=U_{q_a}
+
\begin{cases}
V_{q_a,i/2^a},&2^a\mid i,\\
0,&2^a\nmid i.
\end{cases}
\end{equation}
All sums in this proof are taken over $\mathbb F_2$.

For $a=1$, formula \eqref{eq:024} is exactly
\eqref{eq:022}: here $U_d=e_d$, while for $i=2c$ the second
sum is $e_c+e_{c+d}$.  For the induction step, take $q=q_a$ with
$1\leq a<r$; in particular, $q$ divides $d$ and $m/q$ is even.  For
$1\leq c<q$, formula \eqref{eq:022} gives
\begin{equation*}
U_qB=U_{q/2},\qquad
V_{q,c}B=
\begin{cases}
0,&c\ \text{odd},\\
V_{q/2,c/2},&c\ \text{even}.
\end{cases}
\end{equation*}
Indeed, in $U_qB$ the first and third terms from the even rows give all
nonzero multiples of $q/2$ except $d$, and the middle terms give $e_d$.
For $V_{q,c}B$, the middle terms occur an even number of times and cancel;
the other two terms give the stated residue class when $c$ is even, while
all rows are identical and occur an even number of times when $c$ is odd.
These relations prove \eqref{eq:024} by induction.

At $a=r$ one has $q_r=1$.  Since $m\nmid i$ for $1\leq i<m$, the second
term in \eqref{eq:024} is absent, and every row of $B^r$
is $U_1$.  Hence $B^r=\mathbf J$.

It remains to see that the equality persists.  Every column of $B$ has odd
sum.  For $j<d$ the only contribution is from row $2j$; for $j>d$ it is
from row $2(j-d)$.  Column $d$ receives one contribution from each of the
$d$ odd rows and one from each of the $d-1$ even rows, for a total of
$m-1$ contributions.  Thus $\mathbf JB=\mathbf J$, and consequently
$B^a=\mathbf J$ for all $a\geq r$.
\end{proof}

\subsection{The valuation formula and automaticity}

We next connect the rescaled recurrence to the binomial valuation.  The following
exact formula is the key point: after the predictable power of $2$ is
removed, no division remains.

\begin{proposition}\label{prop:03}
Let $r\geq1$, $m=2^r$, and write
\[
n=mq+\rho,\qquad 0\leq\rho<m.
\]
Put
\[
b(n)=\nu_2\binom{n+m-1}{m-1}.
\]
If $\rho=0$, then $b(n)=0$ and $t_m(n)$ is odd.  If $1\leq\rho<m$, put
\[
h=m-\rho,\qquad
\alpha=r+\nu_2(q+1),\qquad
u=\operatorname{odd}(q+1).
\]
Then
\begin{equation}\label{eq:025}
b(n)=\alpha-w_h
\end{equation}
and
\begin{equation}\label{eq:026}
\operatorname{odd}(t_m(n))
=\bigl(N^\alpha D\mathcal V_u\bigr)_h.
\end{equation}
In particular, \eqref{eq:020} holds.
\end{proposition}

\begin{proof}
Formula~\eqref{eq:006} gives
\[
b(n)=s_2(n)+s_2(m-1)-s_2(n+m-1).
\]
Since $s_2(m-1)=r$, the case $\rho=0$ immediately gives $b(n)=0$.
If $\rho>0$, the low $r$ binary digits separate from the remaining digits,
and hence
\[
\begin{aligned}
b(n)
&=r+s_2(q)+s_2(\rho)-s_2(q+1)-s_2(\rho-1)\\
&=r+\nu_2(q+1)-\nu_2(\rho).
\end{aligned}
\]
Here we used
\[
s_2(q)-s_2(q+1)=\nu_2(q+1)-1,\qquad
s_2(\rho)-s_2(\rho-1)=1-\nu_2(\rho).
\]
Because $0<\rho<m$, one has
\[
w_h=\nu_2(m-\rho)=\nu_2(\rho).
\]
This proves \eqref{eq:025}.

Lemma~\ref{lem:03}, with $e=\ell=1$, gives
\begin{equation*}
D\mathcal V_{2^\alpha u}
=2^\alpha N^\alpha D\mathcal V_u.
\end{equation*}
Since $n+h=m(q+1)=2^\alpha u$, the $h$-th coordinate of this identity is
\begin{equation}\label{eq:027}
t_m(n)
=2^{\alpha-w_h}\bigl(N^\alpha D\mathcal V_u\bigr)_h.
\end{equation}

It remains to show that the last factor is odd.  Put
\[
\mathcal W(x)=\sum_{j=1}^{m-1}2^{w_j}x^j.
\]
By the definition of $w_j$, only odd indices survive modulo $2$, so
\[
\mathcal W(x)
\equiv x+x^3+\cdots+x^{m-1}
=\frac{x(1+x^m)}{1+x^2}
\pmod2.
\]
Also
\[
T(x)^m\equiv(1+x)^{-m}=(1+x^m)^{-1}\pmod2.
\]
Consequently
\begin{equation*}
\mathcal W(x)T(x)^m\equiv\frac{x}{1+x^2}\pmod2.
\end{equation*}
The coefficient of $x^u$ in the left-hand side is
\[
\sum_{j=1}^{m-1}2^{w_j}t_m(u-j),
\]
which is odd because $u$ is odd.  Since $\alpha\geq r$,
Lemma~\ref{lem:09} shows that every coordinate of
$N^\alpha D\mathcal V_u$ is congruent to this sum modulo $2$.  Thus the
factor in \eqref{eq:027} is odd, proving
\eqref{eq:026} and the valuation formula when $\rho>0$.

Finally,
\[
T(x)^m\equiv(1+x^m)^{-1}\pmod2
\]
has coefficient $1$ at every power $x^{mq}$.  Thus $t_m(mq)$ is odd, which
settles the remaining case $\rho=0$.
\end{proof}

\begin{proof}[Proof of Theorem~\ref{thm:pow2}]
Proposition~\ref{prop:03} proves \eqref{eq:020}.

Fix $s\geq1$.  For each $1\leq\rho<m$, formula
\eqref{eq:026} expresses the odd part at
$n=mq+\rho$ as
\begin{equation}\label{eq:028}
\bigl(
N^{\,r+\nu_2(q+1)}
D\mathcal V_{\operatorname{odd}(q+1)}
\bigr)_{m-\rho}
\pmod{2^s}.
\end{equation}
Put $Y_v=D\mathcal V_v$.  The vector sequence
$(Y_v\bmod2^s)_{v\geq0}$ is $2$-automatic by
Lemma~\ref{lem:04}.  The powers of $N$ modulo $2^s$ form an eventually
periodic sequence in the finite matrix ring over
$\mathbb Z/2^s\mathbb Z$.  Hence $C_j=N^{r+j}\bmod2^s$ is eventually
periodic, and Lemma~\ref{lem:05} shows
that both $(C_{\nu_2(q+1)})_{q\geq0}$ and
$(Y_{\operatorname{odd}(q+1)})_{q\geq0}$ are $2$-automatic.  Taking their
finite direct product and then multiplying the two components shows that
\[
\bigl(N^{r+\nu_2(q+1)}Y_{\operatorname{odd}(q+1)}\bmod2^s\bigr)_{q\geq0}
\]
is $2$-automatic.  Extracting the fixed coordinate in
\eqref{eq:028} preserves automaticity.

For $\rho=0$, Proposition~\ref{prop:03} gives
\[
\operatorname{odd}(t_m(mq))=t_m(mq).
\]
This subsequence modulo $2^s$ is $2$-automatic by
Lemma~\ref{lem:04}, since $m=2^r$ and taking the
subsequence indexed by $2^rq$ is an iteration of binary decimation.
Finally, the $m$ residue classes are interleaved by the last $r$ binary
digits of $n$.  Lemma~\ref{lem:06}(1) shows that this
interleaving is $2$-automatic.  The theorem follows.
\end{proof}

\begin{remark}
For $m=2^r$, the entries of $DAD^{-1}$ are divisible by $2$, so one
binary step is enough for the rescaling.  The family $m=3\cdot2^r$
will use two steps, with the integer matrix $DA^2D^{-1}/4$.
\end{remark}

\section{The family \texorpdfstring{$m=3\cdot2^r$}{m=3 times a power of 2}}
\label{sec:05}

This section proves Theorems~\ref{thm:val3} and~\ref{thm:odd3} for the
family whose first member is $m=12$.  We begin with the valuation formula.

\begin{retheorem}{thm:val3}
Let $r\ge2$ and put $m=3\cdot2^r$.  For every $n\ge0$,
\[
\nu_2(t_m(n))=\nu_2\binom{n+m-1}{m-1}.
\]
\end{retheorem}

The proof below treats every $r\geq2$ within the same chain of lemmas, so the
case $r=2$ proves the valuation formula for $m=12$ directly.  There are two
separate points to check when $r=2$.  First, one coefficient-valuation lemma
requires a finite initial calculation; that calculation is included in the
lemma.  Second, an auxiliary correction in the first part of the proof has
a different vanishing threshold.  These two checks are stated and proved
separately below, after all notation involved has been defined.
Put
\[
P=2^r,\qquad m=3P,
\]
and, for $n\geq0$, write
\[
C(n)=\binom{n+m-1}{m-1},
\qquad
b(n)=\nu_2(C(n)).
\]
By \eqref{eq:006},
\begin{equation}\label{eq:043}
b(n)=s_2(n)+s_2(m-1)-s_2(n+m-1).
\end{equation}

We state four auxiliary lemmas, derive the theorem from them, and then
prove the lemmas.

\subsection{Auxiliary lemmas}

The first lemma treats all indices for which the predicted valuation is at
most $r+2$.

\begin{lemma}\label{lem:10}
For $r\geq2$ and $m=3\cdot2^r$, one has
\[
b(n)\leq r+2
\quad\Longrightarrow\quad
\nu_2(t_m(n))=b(n).
\]
\end{lemma}

For the remaining indices, with $b(n)\geq r+3$, we use a diagonally
rescaled two-step recurrence.

Retain the coefficient vector $\mathcal V_n$, the matrix $A$, and the
common weights $w_i$ and $D$ from Section~\ref{sec:01}, now with $m=3P$.
Define the rescaled two-step matrix by
\begin{equation}\label{eq:044}
M=\frac14DA^2D^{-1}.
\end{equation}
The next three lemmas give the weights, the parity of powers of $M$, and the
scalar sum needed after stabilization modulo $2$.

\begin{lemma}\label{lem:11}
Write $1\leq i\leq3P-1$ in one of the forms appearing below.  Then
\[
\begin{array}{c|c}
i&w_i\\ \hline
i=s,\quad 1\leq s<P&\nu_2(s)\\
i=P&r+1\\
i=P+s,\quad1\leq s<P&1+\nu_2(s)\\
i=2P&r+1\\
i=2P+s,\quad1\leq s<P&\nu_2(s).
\end{array}
\]
\end{lemma}

\begin{lemma}\label{lem:12}
The matrix $M$ defined in \eqref{eq:044} has integral entries.
With $\mathbf J$ denoting the $(m-1)\times(m-1)$ matrix whose entries
are all $1$, one has
\begin{equation}\label{eq:045}
M^k\equiv\mathbf J\pmod2
\qquad
\left(k\geq\left\lceil\frac r2\right\rceil+1\right).
\end{equation}
\end{lemma}

\begin{lemma}\label{lem:13}
For $n\geq0$, put
\[
\mathcal S(n)=\sum_{j=1}^{m-1}2^{w_j}t_m(n-j),
\]
where $t_m(a)=0$ for $a<0$.  If $q$ is odd, then
\[
\mathcal S(q)\equiv1\pmod2,
\qquad
\mathcal S(2q)\equiv2\pmod4.
\]
\end{lemma}

\subsection{Proof of the valuation formula}

\begin{proof}[Proof of Theorem~\ref{thm:val3}]
Let $r\geq2$.
\smallskip

\noindent\emph{1. Reduction to the large-valuation range.}
Lemma~\ref{lem:10} proves the result when $b(n)\leq r+2$.
We therefore suppose
\begin{equation}\label{eq:046}
b(n)\geq r+3.
\end{equation}

We first express $b(n)$ in a form that determines how many times the binary
recurrence must be applied.  Write
\[
n=Pq+s,\qquad 0\leq s<P.
\]
Recall that $s_2(m-1)=r+1$.  If $s=0$, then
\[
s_2(n)=s_2(q),
\qquad
s_2(n+m-1)=s_2\bigl(P(q+3)-1\bigr)=s_2(q+2)+r.
\]
Thus \eqref{eq:043} gives
\begin{equation}\label{eq:047}
b(n)=\nu_2\left(\left\lfloor\frac q2\right\rfloor+1\right).
\end{equation}
Indeed, writing $q=2a$ or $q=2a+1$, the digit-sum expression in either
case reduces to
\[
1+s_2(a)-s_2(a+1)=\nu_2(a+1).
\]

If $s>0$, the binary blocks in
\[
n=Pq+s,
\qquad n+m-1=P(q+3)+(s-1)
\]
are disjoint.  Since
$s_2(s)-s_2(s-1)=1-\nu_2(s)$, formula
\eqref{eq:043} becomes
\[
b(n)=r+2-\nu_2(s)+s_2(q)-s_2(q+3).
\]
Write $q=4a+c$, $0\leq c<4$.  Substitution of the four possible final
two binary digits of $q$ gives
\begin{equation}\label{eq:048}
b(n)=
\begin{cases}
r-\nu_2(s),&c=0,\\
r+2-\nu_2(s)+\nu_2(a+1),&c=1\text{ or }3,\\
r+1-\nu_2(s)+\nu_2(a+1),&c=2.
\end{cases}
\end{equation}

\smallskip

\noindent\emph{2. Choosing the coefficient followed by the matrix.}
We choose $h\in\{1,\ldots,m-1\}$ so that
\[
n+h=2^\alpha u,\qquad u\ \text{odd},
\]
and
\begin{equation}\label{eq:049}
b(n)=\alpha-w_h.
\end{equation}
If $s=0$, write $q=2a$ or $q=2a+1$ and choose, respectively,
\[
h=2P\quad\text{or}\quad h=P.
\]
In both cases
\[
n+h=2P(a+1),
\qquad \alpha=r+1+\nu_2(a+1).
\]
Since $w_P=w_{2P}=r+1$, equation \eqref{eq:047} gives
\eqref{eq:049}.

If $s>0$, the case $c=0$ is excluded by
\eqref{eq:046}.  For $c=1,2,3$, choose, respectively,
\[
h=3P-s,\qquad h=2P-s,\qquad h=P-s.
\]
These integers all lie between $1$ and $m-1$, and in each case
\[
n+h=4P(a+1),
\qquad \alpha=r+2+\nu_2(a+1).
\]
Moreover, $\nu_2(P-s)=\nu_2(s)$, so the three corresponding weights are
\[
w_h=\nu_2(s),\qquad1+\nu_2(s),\qquad\nu_2(s),
\]
respectively.  Comparing with \eqref{eq:048} proves
\eqref{eq:049}.  In every case
\[
\alpha=b(n)+w_h\geq r+3.
\]
Since $t_m(n)$ is the $h$-th coordinate of
$\mathcal V_{n+h}=\mathcal V_{2^\alpha u}$, the integer $\alpha$ counts
the binary recurrence steps and $w_h$ is the weight of the chosen
coordinate.

\smallskip

\noindent\emph{3. Applying the stabilized two-step matrix.}
The iteration used in the proof of Lemma~\ref{lem:03},
with $e=\ell=2$, gives
\begin{equation}\label{eq:050}
DA^{2k}=2^{2k}M^kD\qquad(k\geq0).
\end{equation}
If $\alpha=2k$ is even, all applications of the binary recurrence can be grouped into $k$
two-step blocks, and
\[
D\mathcal V_{2^\alpha u}
=DA^{2k}\mathcal V_u
=2^{2k}M^kD\mathcal V_u.
\]
The $h$-th coordinate and \eqref{eq:049} give
\[
t_m(n)=2^{b(n)}(M^kD\mathcal V_u)_h.
\]
Since $\alpha\geq r+3$, one has
$k\geq\lceil r/2\rceil+1$.  Lemma~\ref{lem:12} and
Lemma~\ref{lem:13} therefore give
\[
(M^kD\mathcal V_u)_h
\equiv\sum_{j=1}^{m-1}2^{w_j}t_m(u-j)
=\mathcal S(u)
\equiv1\pmod2.
\]
Thus $\nu_2(t_m(n))=b(n)$.

Suppose now that $\alpha=2k+1$.  One application of the binary recurrence remains, so the $k$
two-step blocks start from $2u$:
\[
\mathcal V_{2^\alpha u}=A^{2k}\mathcal V_{2u}.
\]
Then
\[
D\mathcal V_{2^\alpha u}=2^{2k}M^kD\mathcal V_{2u}.
\]
Since $2k+1=\alpha\geq r+3$, one has
$k\geq\lceil r/2\rceil+1$; hence
Lemma~\ref{lem:12} applies to $M^k$.
Every coordinate of $D\mathcal V_{2u}$ is even.  Indeed, this is immediate when
$w_j\geq1$; when $w_j=0$, Lemma~\ref{lem:11} shows that
$j$ is odd, and the coefficient $t_m(2u-j)$ is even because
\[
T(x)^m\equiv(1+x^P)^{-3}\pmod2
\]
contains no odd powers.  Since $M^k-\mathbf J$ is an even integral
matrix, multiplication by the even vector $D\mathcal V_{2u}$ makes its
contribution divisible by $4$.  Hence
\[
(M^kD\mathcal V_{2u})_h
\equiv\mathcal S(2u)\equiv2\pmod4
\]
by Lemmas~\ref{lem:12} and
\ref{lem:13}.  Therefore
\[
\nu_2\bigl((M^kD\mathcal V_{2u})_h\bigr)=1,
\]
and the $h$-th coordinate gives the exact identity
\[
t_m(n)=2^{2k-w_h}(M^kD\mathcal V_{2u})_h.
\]
Consequently
\[
\nu_2(t_m(n))=2k+1-w_h=\alpha-w_h=b(n).
\]
This completes the proof.
\end{proof}

\subsection{Proofs of the auxiliary lemmas}

\subsubsection{The small-valuation range}

We first treat the indices for which $b(n)\leq r+2$.  Define
\[
V(y)=(1-y)T(y).
\]
Since
\[
(1+x)T(x)=(1-x^2)T(x^2)=V(x^2),
\]
we have
\begin{equation}\label{eq:051}
T(x)^m=(1+x)^{-m}V(x^2)^m.
\end{equation}
Over $\mathbb F_2[[y]]$ one has
\[
T(y)=\prod_{j\geq0}(1+y^{2^j})=\frac1{1+y}.
\]
Consequently $V(y)\equiv1\pmod2$, and there is a uniquely determined
$W(y)\in y\mathbb Z[[y]]$ such that
\begin{equation}\label{eq:052}
V(y)=1+2W(y).
\end{equation}

The proof uses the three successive moduli
$2^{r+1}$, $2^{r+2}$, and $2^{r+3}$.  The first lemma deals with the
first two; the next two identify and then eliminate the last correction.

\begin{lemma}\label{lem:14}
For every $n\geq0$,
\begin{align}
t_m(n)&\equiv(-1)^nC(n)\pmod {2^{r+1}},
\label{eq:053}\\
b(n)\geq r
&\quad\Longrightarrow\quad
t_m(n)\equiv(-1)^nC(n)\pmod {2^{r+2}}.\notag
\end{align}
Consequently $\nu_2(t_m(n))=b(n)$ whenever $b(n)\leq r+1$.
\end{lemma}

\begin{proof}

First note the following elementary fact.  If $Z\in2\mathbb Z[[y]]$, then
\begin{equation}\label{eq:054}
(1+Z)^{2^r}\equiv1\pmod {2^{r+1}}.
\end{equation}
Indeed, the assertion for $r=1$ follows by expanding $(1+Z)^2$.
If $(1+Z)^{2^r}=1+2^{r+1}U$, then
\[
(1+Z)^{2^{r+1}}
=1+2^{r+2}\bigl(U+2^r U^2\bigr),
\]
which proves \eqref{eq:054} by induction.  Applying it to
$Z=2W(y)$ and then taking the third power gives
\[
V(y)^m\equiv1\pmod {2^{r+1}}.
\]
It follows from \eqref{eq:051} that
the first congruence of the lemma holds.
Thus the desired valuation formula already follows when $b(n)\leq r$.

We next keep one more power of $2$.  In the expansion
\[
(1+2W)^m=\sum_{k=0}^m2^k\binom{m}{k}W^k,
\]
the terms with $k\geq3$ are divisible by $2^{r+2}$.  To see this, use
\[
\binom{m}{k}=\frac{m}{k}\binom{m-1}{k-1}
\]
to obtain
\[
\nu_2\left(2^k\binom{m}{k}\right)
\geq r+k-\nu_2(k)\geq r+2
\qquad(k\geq3).
\]
The first two nonconstant terms therefore give
\begin{equation}\label{eq:055}
V(y)^m
\equiv1+2^{r+1}\bigl(W(y)+W(y)^2\bigr)
\pmod {2^{r+2}}.
\end{equation}

We claim that
\begin{equation}\label{eq:056}
W(y)+W(y)^2\equiv\frac{y}{1-y}\pmod2.
\end{equation}
Write $T(y)=\sum_{a\geq0}\epsilon(a)y^a$, where
$\epsilon(a)=(-1)^{s_2(a)}$.  From
$V(y)=(1-y)T(y)$, the coefficient of $y^a$ in $W(y)$ is
\[
\frac{\epsilon(a)-\epsilon(a-1)}2
\qquad(a\geq1).
\]
Since
\[
s_2(a-1)=s_2(a)-1+\nu_2(a),
\]
one has $\epsilon(a-1)=(-1)^{\nu_2(a)-1}\epsilon(a)$.  Thus the displayed
coefficient is odd precisely when $\nu_2(a)$ is even.  Over $\mathbb F_2$,
the coefficient of $y^a$ in $W(y)^2$ is zero for odd $a$ and is the
coefficient of $y^{a/2}$ in $W(y)$ for even $a$.  Hence the coefficient of
$y^a$ in $W+W^2$ is $1$ for every $a\geq1$, proving
\eqref{eq:056}.

Let
\[
A(n)=[x^n](1+x)^{-m}\frac{x^2}{1-x^2}.
\]
Equations \eqref{eq:051}, \eqref{eq:055}, and
\eqref{eq:056} show that
\begin{equation}\label{eq:057}
\frac{t_m(n)-(-1)^nC(n)}{2^{r+1}}
\equiv A(n)\pmod2.
\end{equation}
We now show
\begin{equation}\label{eq:058}
A(n)\ \text{odd}\quad\Longrightarrow\quad b(n)\leq r-1.
\end{equation}
Modulo $2$,
\[
(1+x)^{-m}=(1+x^P)^{-3}.
\]
Thus $A(n)=0$ modulo $2$ when $n$ is odd.  If $n=2L$, put
$h=P/2$.  Then
\[
A(2L)\equiv
[y^{L-1}]
\frac1{(1+y)(1+y^h)(1+y^{2h})}
\pmod2.
\]
Set $Q=\lfloor(L-1)/h\rfloor$.  The coefficient on the right counts modulo
$2$ the pairs $(u,v)\in\mathbb Z_{\geq0}^2$ satisfying
$u+2v\leq Q$.  If $Q=2a$, their number is $(a+1)^2$; if $Q=2a+1$, their
number is $(a+1)(a+2)$.  It is therefore odd exactly when
$Q\equiv0\pmod4$.

Suppose this happens.  Write
\[
L-1=4ah+v,\qquad 0\leq v<h.
\]
Then
\[
n=4aP+2v+2.
\]
Put $d=2v+2$.  If $d<P$, the binary digits in
$n=4aP+d$ occupy disjoint blocks, as do those in
\[
n+3P-1=P(4a+3)+(d-1).
\]
Using $s_2(4a+3)=s_2(a)+2$ and
$s_2(d)-s_2(d-1)=1-\nu_2(d)$ in \eqref{eq:043} gives
\[
b(n)=r-\nu_2(d)\leq r-1.
\]
If $d=P$, then $n=P(4a+1)$ and
$n+3P-1=4P(a+1)-1$; substitution in
\eqref{eq:043} gives $b(n)=0$.  This proves
\eqref{eq:058}.  Hence
\eqref{eq:057} is divisible by $2$ whenever $b(n)\geq r$.
In particular, the valuation formula follows for $b(n)\leq r+1$.
This proves Lemma~\ref{lem:14}.
\end{proof}

We now identify the correction that appears one power of $2$ later.
The function $\chi$ in the next lemma is the nontrivial character
modulo $4$ evaluated at the odd part of its argument.

\begin{lemma}\label{lem:15}
For $a\geq1$, define
\[
\chi(a)=
\begin{cases}
1,&a/2^{\nu_2(a)}\equiv1\pmod4,\\
-1,&a/2^{\nu_2(a)}\equiv3\pmod4.
\end{cases}
\]
Then
\[
\frac{V(y)^m-1}{2^{r+1}}
\equiv\sum_{a\geq1}\chi(a)y^a\pmod4.
\]
\end{lemma}

\begin{proof}
The congruence already proved from \eqref{eq:054} shows that every
coefficient of $V(y)^m-1$ is divisible by $2^{r+1}$.  Define, only for
this proof,
\[
H(y)=\frac{V(y)^m-1}{2^{r+1}}\in\mathbb Z[[y]].
\]
Then
\begin{equation*}
V(y)^m=1+2^{r+1}H(y),
\end{equation*}
and the binomial expansion modulo $2^{r+3}$ gives
\begin{equation*}
H(y)\equiv3W(y)+W(y)^2+2W(y)^4\pmod4.
\end{equation*}
Indeed, after division by $2^{r+1}$, the terms with $k=1,2,4$ are,
respectively,
\[
\begin{aligned}
k=1:&\quad 3W,\\
k=2:&\quad 3(m-1)W^2\equiv W^2\pmod4,\\
k=4:&\quad
(m-1)(m-2)(m-3)W^4\equiv2W^4\pmod4,
\end{aligned}
\]
because $m\equiv0\pmod4$.  For $3\leq k\leq m$ with $k\ne4$, one has
\[
\nu_2\left(2^k\binom{m}{k}\right)
\geq r+k-\nu_2(k)\geq r+3;
\]
here the last inequality holds for $k=3$ and every $k\geq5$.
Thus all remaining terms vanish modulo $4$ after division by
$2^{r+1}$.  This also verifies explicitly that the expansion remains
valid when $r=2$.

We determine $H$ explicitly.  From
\[
V(y)=\frac{1-y}{1+y}V(y^2)
\]
and \eqref{eq:052}, if
\[
Z=W(y^2),\qquad p=\frac{y}{1+y},
\]
then
\[
W(y)=Z-p-2pZ.
\]
For $f(X)=3X+X^2+2X^4$, direct expansion modulo $4$ gives
\[
f(W(y))-f(Z)\equiv p+p^2+2p^4\pmod4.
\]
Indeed, put $U=Z-p$.  Since $W=U-2pZ$, one has, modulo $4$,
\[
3W\equiv3U+2pZ,
\qquad W^2\equiv U^2,
\qquad 2W^4\equiv2U^4.
\]
Moreover, $U^4\equiv Z^4+p^4\pmod2$.  Substitution therefore gives
\[
\begin{aligned}
f(W)
&\equiv3(Z-p)+2pZ+(Z-p)^2+2(Z^4+p^4)\\
&\equiv 3Z+Z^2+2Z^4+p+p^2+2p^4
 =f(Z)+p+p^2+2p^4\pmod4,
\end{aligned}
\]
as asserted.
Furthermore,
\[
p+p^2+2p^4-\frac{y}{1+y^2}
=\frac{4y^4(1+y+y^2)}{(1+y)^4(1+y^2)}.
\]
Therefore
\begin{equation}\label{eq:059}
H(y)\equiv H(y^2)+\frac{y}{1+y^2}\pmod4.
\end{equation}
The recurrence \eqref{eq:059} says that the coefficient
$h_a$ of $y^a$ in $H$ satisfies
\[
h_{2a}\equiv h_a\pmod4,
\qquad h_{2a+1}\equiv(-1)^a\pmod4.
\]
Consequently
\begin{equation*}
H(y)\equiv\sum_{a\geq1}\chi(a)y^a\pmod4.
\end{equation*}
This proves Lemma~\ref{lem:15}.
\end{proof}

The next lemma isolates the finite calculation that will be used at the
boundary of the convolution.  Its first row controls ordinary partial
sums; its last display controls the case in which the remainder is zero.

\begin{lemma}\label{lem:16}
Put
\[
a_k=(-1)^k\binom{k+5}{5}\qquad(k\geq0)
\]
and
\[
J_q=\sum_{k=0}^{q-1}a_k\chi(q-k)\qquad(q\geq0),
\]
where the empty sum is $J_0=0$.  Modulo $4$, $(a_k)$ has period $16$,
and one period together with its partial sums is
\begin{equation}\label{eq:060}
\begin{array}{c|rrrrrrrrrrrrrrrr}
k&0&1&2&3&4&5&6&7&8&9&10&11&12&13&14&15\\ \hline
a_k&1&2&1&0&2&0&2&0&3&2&3&0&0&0&0&0\\
\sum_{j=0}^{k}a_j&1&3&0&0&2&2&0&0&3&1&0&0&0&0&0&0
\end{array}
\end{equation}
Furthermore,
\begin{equation}\label{eq:061}
\begin{aligned}
J_{8v+3}\equiv J_{8v+7}
&\equiv
\begin{cases}2,&v\ \mathrm{even},\\0,&v\ \mathrm{odd},\end{cases}\\
J_{8v+4}\equiv J_{8v+6}
&\equiv
\begin{cases}2,&v\equiv1,2\pmod4,\\0,&v\equiv0,3\pmod4,\end{cases}\\
J_{8v+5}&\equiv0
\end{aligned}
\pmod4\qquad(v\geq0).
\end{equation}
\end{lemma}

\begin{proof}
Vandermonde's identity gives
\[
\binom{N+16}{5}-\binom N5
=\sum_{j=1}^{5}\binom{16}{j}\binom N{5-j}.
\]
The right-hand side is divisible by $4$, and the sign in $a_k$ is
unchanged by a shift of $16$.  This proves the period; direct evaluation
of the first sixteen terms gives the two rows of
\eqref{eq:060}.  In particular, a complete period
has sum $0$.

For $J_q$, splitting the convolution after a shift of $16$ gives
\begin{equation*}
J_{q+16}-J_q
\equiv\sum_{k=0}^{15}a_k\chi(q+16-k)\pmod4.
\end{equation*}
Only $k=0,1,2,4,6,8,9,10$ contribute.  Substitution, using
\[
\chi(8v)=\chi(v),\qquad
(\chi(8v+1),\ldots,\chi(8v+7))
=(1,1,-1,(-1)^v,1,-1,-1),
\]
gives
\[
\begin{array}{c|cc|c}
c&J_c\bmod4&J_{8+c}\bmod4&(J_{8v+c+16}-J_{8v+c})\bmod4\\ \hline
3&2&0&0\\
4&0&2&2\\
5&0&0&0\\
6&0&2&2\\
7&2&0&0.
\end{array}
\]
The two columns following $c$ give initial residues modulo $4$;
the last column and
induction on $v$ give \eqref{eq:061}.
\end{proof}

The final auxiliary lemma translates the periodic calculation into the
vanishing needed for the valuation argument: the character correction
disappears whenever the binomial coefficient already has sufficiently
large valuation.

\begin{lemma}\label{lem:17}
Let
\[
R(n)=[x^n](1+x)^{-m}\sum_{a\geq1}\chi(a)x^{2a}.
\]
Then
\begin{equation}\label{eq:062}
R(n)\equiv0\pmod4
\quad\text{if}\quad
\begin{cases}
b(n)\geq r+1,&r\geq3,\\
b(n)\geq4,&r=2.
\end{cases}
\end{equation}
\end{lemma}

The separate threshold for $r=2$ is necessary.  Indeed, when $r=2$ and
$n=6$, one has
\[
b(6)=\nu_2\binom{17}{11}=3,
\qquad
R(6)=\binom{15}{11}+\binom{13}{11}-1\equiv2\pmod4.
\]
Thus the condition $b(n)\geq r+1$ valid for $r\geq3$ cannot simply be
extended to $r=2$.

\begin{proof}
First,
\[
(1+x)^{P}\equiv(1+x^{P/2})^2\pmod4,
\]
because, for $0<j<P$,
\[
\nu_2\binom Pj=r-\nu_2(j).
\]
Indeed, use
$\binom Pj=(P/j)\binom{P-1}{j-1}$.  Legendre's formula shows that
$\binom{P-1}{j-1}$ is odd, because the $r$ binary digits of $P-j$ are
obtained from those of $j-1$ by interchanging $0$ and $1$, so
$s_2(j-1)+s_2(P-j)=r$.  Thus among the interior coefficients of
$(1+x)^P$, only $\binom{P}{P/2}$ has valuation $1$.  Since both sides
have constant term $1$, their inverses are congruent modulo $4$, and hence
\begin{equation}\label{eq:063}
(1+x)^{-m}\equiv(1+x^{P/2})^{-6}\pmod4.
\end{equation}
Thus $R(n)\equiv0\pmod4$ for odd $n$, since $P/2$ is even.
For $n=2L$, put $d=P/4$ and obtain
\begin{equation}\label{eq:064}
R(2L)\equiv
[y^L](1+y^d)^{-6}\sum_{j\geq1}\chi(j)y^j\pmod4.
\end{equation}
Write $L=dq+u$, where $0\leq u<d$.

When $r=2$, one has $d=1$, so only the case $u=0$ below can occur.
The first two cases below are therefore needed only when $r\geq3$.

\smallskip
\noindent\emph{Case 1: $0<u<d$ and $u\ne d/2$.}
Put $e=\nu_2(u)$.  The excluded case $u=d/2$ is precisely the only case in
which $e=\nu_2(d)-1$; hence
$e\leq\nu_2(d)-2$.  It follows that
\[
\frac{d(q-k)+u}{2^e}\equiv\frac{u}{2^e}\pmod4,
\]
and therefore $\chi(d(q-k)+u)=\chi(u)$ for $0\leq k\leq q$.
Therefore
\begin{equation}\label{eq:065}
R(2L)\equiv
\chi(u)\sum_{k=0}^{q}(-1)^k\binom{k+5}{5}\pmod4.
\end{equation}
By the partial-sum row of
Lemma~\ref{lem:16}, a nonzero value in
\eqref{eq:065} is possible only when
\[
q\equiv0,1,4,5,8,9\pmod {16}.
\]
If $q=16v+c$ for one of these six residues, then $c+6<16$, so adding $6$
does not affect the binary digits of $v$.  The six corresponding
differences are
\[
\begin{array}{c|rrrrrr}
c&0&1&4&5&8&9\\ \hline
s_2(c)-s_2(c+6)&-2&-2&-1&-1&-2&-2.
\end{array}
\]
Thus $s_2(q)-s_2(q+6)\leq-1$.  To compute $b(2L)$, note that
$2L=2^{r-1}q+2u$ and
\[
2L+3P-1=2^{r-1}(q+6)+(2u-1).
\]
The two summands on each right-hand side occupy disjoint blocks of binary
digits.  Since
$s_2(2u-1)=s_2(u-1)+1=s_2(u)+\nu_2(u)$, formula
\eqref{eq:043} gives
\begin{equation*}
b(2L)=r+1+s_2(q)-s_2(q+6)-\nu_2(u).
\end{equation*}
Hence $R(2L)\ne0\pmod4$ implies $b(2L)\leq r$.

\smallskip
\noindent\emph{Case 2: $u=d/2$.}
Then
\[
\chi(d(q-k)+d/2)=(-1)^{q-k}.
\]
The hockey-stick identity gives
\[
R(2L)\equiv(-1)^q\binom{q+6}{6}\pmod4.
\]
In this case $2L=2^{r-2}(2q+1)$ and
$2L+3P-1=2^{r-2}(2q+13)-1$.  Consequently
\[
b(2L)=4+s_2(q)-s_2(q+6),
\]
by \eqref{eq:043}.  Legendre's formula also gives
\[
\nu_2\binom{q+6}{6}=2+s_2(q)-s_2(q+6).
\]
Thus $R(2L)\ne0\pmod4$ again implies $b(2L)\leq3\leq r$.

\smallskip
\noindent\emph{Case 3: $u=0$.}
Expanding $(1+y)^{-6}$ in
\eqref{eq:064} and using the definitions in
Lemma~\ref{lem:16} gives directly
\[
R(2L)\equiv J_q\pmod4.
\]
When $u=0$, one has $2L=2^{r-1}q$ and
$2L+3P-1=2^{r-1}(q+6)-1$.  Hence \eqref{eq:043} gives
\[
b(2L)=2+s_2(q)-s_2(q+5).
\]
For $q\equiv0,1,2\pmod8$, the difference
$s_2(q)-s_2(q+5)$ is, respectively, $-2,-1,-2$.  If
$q=8v+c$ with $c=3,4,5,6,7$, then it is, respectively,
\[
\nu_2(v+1)+1,\quad
\nu_2(v+1)-1,\quad
\nu_2(v+1),\quad
\nu_2(v+1)-1,\quad
\nu_2(v+1)+1.
\]
The sections \eqref{eq:061} show in every case that
$J_q\ne0\pmod4$ implies
\[
s_2(q)-s_2(q+5)\leq1.
\]
Hence $J_q\ne0\pmod4$ implies $b(2L)\leq3$.  For $r=2$ this proves that
$R(2L)\equiv0\pmod4$ whenever $b(2L)\geq4$; for $r\geq3$ it proves that
$R(2L)\equiv0\pmod4$ whenever $b(2L)\geq r+1$.  This completes the proof
of \eqref{eq:062}.
\end{proof}

\begin{proof}[Proof of Lemma~\ref{lem:10}]
Lemma~\ref{lem:14} proves the assertion for
$b(n)\leq r+1$.  It remains to consider $b(n)=r+2$.
Taking the coefficient of $x^n$ in \eqref{eq:051} gives
\[
\frac{t_m(n)-(-1)^nC(n)}{2^{r+1}}
=[x^n](1+x)^{-m}
\frac{V(x^2)^m-1}{2^{r+1}}.
\]
By Lemma~\ref{lem:15}, the right-hand side is congruent
to $R(n)$ modulo $4$.  Lemma~\ref{lem:17} applies:
for $r\geq3$ one has $b(n)=r+2\geq r+1$, and for $r=2$ one has
$b(n)=4$.  Hence the error is divisible by
$2^{r+3}=2^{b(n)+1}$ and cannot change the valuation of
$(-1)^nC(n)$.  This proves the remaining case.
\end{proof}

\subsubsection{The coordinate weights}

\begin{proof}[Proof of Lemma~\ref{lem:11}]
Since $m-2=3P-2$, its binary expansion has a $1$ in position $r+1$, a
$0$ in position $r$, $1$'s in positions $1,\ldots,r-1$, and a $0$ in
position $0$.  Hence $s_2(m-2)=r$.  Legendre's formula gives
\[
w_i=s_2(i-1)+s_2(m-1-i)-r.
\]
For $1\leq s<P$, we shall use
\[
s_2(s-1)=s_2(s)-1+\nu_2(s),
\qquad s_2(P-1-s)=r-s_2(s).
\]
If $i=s$, then
\[
s_2(i-1)=s_2(s-1),\qquad
s_2(m-1-i)=1+r-s_2(s),
\]
so
\[
w_i=1+s_2(s-1)-s_2(s)=\nu_2(s).
\]
If $i=P+s$, then
\[
s_2(i-1)=1+s_2(s-1),
\qquad s_2(m-1-i)=1+r-s_2(s),
\]
and hence $w_i=1+\nu_2(s)$.  If $i=2P+s$, then
\[
s_2(i-1)=1+s_2(s-1),
\qquad s_2(m-1-i)=r-s_2(s),
\]
and hence $w_i=\nu_2(s)$.  Finally, at $i=P$ and $i=2P$, the two digit
sums in the formula for $w_i$ are $r$ and $r+1$, in either order.
Thus $w_P=w_{2P}=r+1$.
\end{proof}

\subsubsection{The rescaled two-step matrix}

The proof of Lemma~\ref{lem:12} has three auxiliary
parts: an exact coefficient valuation, the support of the powers of $M$
modulo $2$, and the terminal column counts.

The next calculation determines the coefficient valuations needed for
the matrix support when $m=3P$.

\begin{lemma}\label{lem:18}
Put
\[
\mathcal R(x)=(1-x)^6(1+x)^3
=1-3x+8x^3-6x^4-6x^5+8x^6-3x^8+x^9
\]
and write
\[
\mathcal R(x)^P=\sum_{\ell=0}^{9P}d_\ell x^\ell.
\]
For indices outside $0\leq\ell\leq9P$, set $d_\ell=0$.
For $0<s<P$ and $0\leq q\leq8$,
\begin{equation}\label{eq:066}
\nu_2(d_{qP+s})=r-\nu_2(s)+c_q,
\end{equation}
where
\[
(c_0,\ldots,c_8)=(0,3,2,3,1,3,2,3,0).
\]
At multiples of $P$,
\begin{equation}\label{eq:067}
\bigl(\nu_2(d_{qP})\bigr)_{q=0}^{9}
=(0,0,2,2,1,1,2,2,0,0).
\end{equation}
\end{lemma}

\begin{proof}
For $r=2$, one has $P=4$ and
\[
\mathcal R(x)^P=(1-x)^{12}(1-x^2)^{12}.
\]
Direct coefficient extraction gives
\begin{equation*}
d_\ell=
\sum_{\substack{0\leq v\leq12\\0\leq\ell-2v\leq12}}
(-1)^{\ell-v}\binom{12}{\ell-2v}\binom{12}{v}.
\end{equation*}
The corresponding integer coefficients are listed in
Appendix~\ref{app:r2data}.  Their $2$-adic valuations, grouped by indices
modulo $4$, are
\begin{equation*}
\begin{aligned}
&(0,0,2,2,1,1,2,2,0,0),\\
&(2,5,4,5,3,5,4,5,2),\\
&(1,4,3,4,2,4,3,4,1),\\
&(2,5,4,5,3,5,4,5,2).
\end{aligned}
\end{equation*}
The first row is exactly \eqref{eq:067}.  Since
\[
2-\nu_2(1)=2,\qquad 2-\nu_2(2)=1,\qquad 2-\nu_2(3)=2,
\]
the last three rows give \eqref{eq:066}.  Thus the lemma holds
for $r=2$.

Assume from now on that $r\geq3$.  Define
\[
B(z)=\sum_{i=0}^{8}b_iz^i,
\qquad
(b_0,\ldots,b_8)=(1,8,12,8,6,8,12,8,1),
\]
and put $G(z)=B(z)B(z^8)$.  We first prove that
\begin{equation}\label{eq:068}
\mathcal R(z)^8\equiv G(z)\pmod {16}.
\end{equation}
Indeed,
\[
\mathcal R(z)\equiv(1+z)(1+z^8)\pmod2.
\]
If $F(z)=(1+z)(1+z^8)$, every term in the binomial expansion of
$\mathcal R(z)^8-F(z)^8$ is divisible by $16$: for the term containing
$(\mathcal R-F)^j$, this follows from
\[
\nu_2\binom8j+j\geq4\qquad(1\leq j\leq8).
\]
Finally,
\[
(1+z)^8\equiv B(z)\pmod {16},
\]
and the same identity with $z^8$ in place of $z$ proves
\eqref{eq:068}.  Moreover,
\[
\frac{\mathcal R(z)^{16}}{\mathcal R(z^2)^8}
=\left(\frac{(1-z)^2}{1+z^2}\right)^{24}
=\left(1-\frac{2z}{1+z^2}\right)^{24}
\equiv1\pmod {16},
\]
where \eqref{eq:054} with $r=3$ shows that the eighth power of
the last parenthesis is $1$ modulo $16$, and its twenty-fourth power is
the cube of that eighth power.
Together with \eqref{eq:068}, this proves
\[
G(z)^2\equiv G(z^2)\pmod {16}.
\]
Starting from \eqref{eq:068}, square the congruence and use the last
identity at each step.  Induction on $r$ gives
\begin{equation}\label{eq:069}
\mathcal R(x)^P\equiv G(x^{P/8})\pmod {16}.
\end{equation}

Set $b_{-1}=b_9=0$.  Since $G(z)=B(z)B(z^8)$, the coefficient of
$z^{8q}$ is $b_q+b_{q-1}$.  Thus these ten coefficients are
\[
(1,9,4,4,14,14,4,4,9,1),
\]
whose valuations are all smaller than $4$.  Since
\eqref{eq:069} is a congruence modulo $16$, these are also the exact
valuations of the corresponding coefficients of $\mathcal R(x)^P$.
This proves \eqref{eq:067}.

For \eqref{eq:066}, differentiate
$\mathcal R(x)^P$.  If $n=qP+s$ with $0<s<P$, then
\[
n d_n
=P[x^{n-1}]\mathcal R'(x)\mathcal R(x)^{P-1}.
\]
Since $\nu_2(n)=\nu_2(s)$,
\begin{equation}\label{eq:070}
\nu_2(d_n)
=r-\nu_2(s)
+\nu_2\bigl([x^{n-1}]\mathcal R'(x)\mathcal R(x)^{P-1}\bigr).
\end{equation}
Moreover,
\[
\frac{\mathcal R'(x)}{\mathcal R(x)}
=-3\frac{1+3x}{1-x^2}.
\]
We next evaluate the partial sums of $G$ algebraically.  For
$0\leq q\leq8$ and $0\leq t\leq7$, the product form of $G$ gives
\begin{equation}\label{eq:071}
\sum_{a=0}^{8q+t}[z^a]G(z)
=B(1)\sum_{j=0}^{q-1}b_j+b_q\sum_{i=0}^{t}b_i.
\end{equation}
Here $B(1)=64$, while
\[
\left(\sum_{i=0}^{t}b_i\right)_{t=0}^{7}
=(1,9,21,29,35,43,55,63)
\]
consists entirely of odd integers.  Hence the valuation modulo $16$ of
the left-hand side of \eqref{eq:071} is
$\nu_2(b_q)=c_q$.  Since $c_q\leq3$, any integer congruent to this partial
sum modulo $16$ has the same exact valuation $c_q$.

For $r\geq4$, the integer $P/8$ is even.  The coefficients of
\[
-3\frac{1+3x}{1-x^2}
\]
are $-3$ in even degrees and $-9$ in odd degrees.  Thus all terms that
occur in the convolution with $G(x^{P/8})$ have the same odd multiplier.
By \eqref{eq:069}, the coefficient in the last term of
\eqref{eq:070} is therefore, modulo $16$, an odd
multiple of
\[
\sum_{a=0}^{8q+\lfloor8(s-1)/P\rfloor}[z^a]G(z).
\]
Equation \eqref{eq:071} shows that its valuation is $c_q$.

It remains to justify the same conclusion when $r=3$, so $P=8$.  Let
\[
L(x)=-3\frac{1+3x}{1-x^2},
\qquad K(x)=B(x)L(x).
\]
For every degree $v\geq8$, the coefficient of $x^v$ in $K(x)$ is zero
modulo $16$, because
\[
\sum_{\substack{0\leq i\leq8\\i\text{ even}}}b_i
=\sum_{\substack{0\leq i\leq8\\i\text{ odd}}}b_i=32.
\]
For $0\leq v\leq6$, that coefficient is odd: the term with $i=0$ is
odd, and every other contributing $b_i$ is even.  Since
$G(x)=B(x)B(x^8)$, it follows that, for $1\leq s<8$,
\[
[x^{8q+s-1}]G(x)L(x)
\equiv b_q[x^{s-1}]K(x)\pmod {16}.
\]
Its valuation is again $\nu_2(b_q)=c_q$.  In all cases
\eqref{eq:070} now gives
\eqref{eq:066}.
\end{proof}

The stabilization of $M$ has three separate ingredients: the support of
$M^k$ modulo $2$, three terminal column counts, and the final collapse to
the all-one matrix.  We record them separately.  The sets below are used
only in the first ingredient; each one merely lists two or four residue
classes.

For a positive integer $q$ and $0\leq c\leq2q$, define subsets of
$\mathbb Z/(4q)\mathbb Z$ by
\[
\begin{aligned}
E^{(q)}_0&=\{0,q\},\\
E^{(q)}_c&=\{0,c,q,q+c\} &&(1\leq c<q),\\
E^{(q)}_q&=\{0,2q\},\\
E^{(q)}_{q+c}&=\{0,2q+c,3q,3q+c\} &&(1\leq c<q),\\
E^{(q)}_{2q}&=\{0,3q\}.
\end{aligned}
\]
For an integer $k\geq1$ with $4^k\leq P$, put $q=P/4^k$ and,
for $1\leq i<3P$, put
\[
\gamma_k(i)=
\begin{cases}
i/4^k,&i<P\text{ and }4^k\mid i,\\
0,&i<P\text{ and }4^k\nmid i,\\
q,&P\leq i\leq2P,\\
q+(i-2P)/4^k,&i>2P\text{ and }4^k\mid i,\\
2q,&i>2P\text{ and }4^k\nmid i.
\end{cases}
\]

\begin{lemma}\label{lem:19}
The matrix $M$ has integral entries.  Moreover, for every integer
$k\geq1$ with $4^k\leq P$, putting $q=P/4^k$ gives
\begin{equation}\label{eq:072}
(M^k)_{ij}\equiv1\pmod2
\quad\Longleftrightarrow\quad
j\bmod 4q\in E^{(q)}_{\gamma_k(i)}.
\end{equation}
\end{lemma}

\begin{proof}
Applying \eqref{eq:008} twice expresses
$t_m(4n-i)$ in terms of the coefficients $t_m(n-j)$ with matrix $A^2$.
On the other hand,
\[
T(x)^m=(1-x)^m(1-x^2)^mT(x^4)^m.
\]
Comparing the coefficient multiplying $t_m(n-j)$ gives
\begin{equation*}
(A^2)_{ij}
=[x^{4j-i}](1-x)^m(1-x^2)^m
=[x^{4j-i}]\mathcal R(x)^P
=d_{4j-i}.
\end{equation*}
Thus
\[
M_{ij}=2^{w_i-w_j-2}d_{4j-i}.
\]
We first prove \eqref{eq:072} for $k=1$ and at the same time
prove that $M$ is
integral.  If $d_{4j-i}\ne0$, write $4j-i=vP+s$ with $0\leq s<P$.
Then $0\leq v\leq8$ when $s>0$, and $0\leq v\leq9$ when $s=0$, so the
two parts of Lemma~\ref{lem:18} cover every case.
For $s>0$, Lemmas~\ref{lem:11} and
\ref{lem:18} give
\begin{equation}\label{eq:073}
\nu_2(M_{ij})=w_i-w_j-2+r-\nu_2(s)+c_v,
\end{equation}
whereas for $s=0$ they give
\begin{equation}\label{eq:074}
\nu_2(M_{ij})=w_i-w_j-2+\nu_2(d_{vP}).
\end{equation}
Use $\nu_2(P-s)=\nu_2(s)$ for $0<s<P$---after division by
$2^{\nu_2(s)}$, both remaining factors are odd---and
$\nu_2(4s)=2+\nu_2(s)$ in these two formulas.  Splitting $i$ according to
the five lines in the definition of $\gamma_1$ gives
\[
\begin{array}{c|c|c}
\gamma_1(i)&j\bmod P\text{ for }\nu_2(M_{ij})=0
&\text{all other }j\\ \hline
0&0,P/4&\nu_2(M_{ij})\geq1\text{ or }d_{4j-i}=0\\
c,\ 1\leq c<P/4&0,c,P/4,P/4+c
&\nu_2(M_{ij})\geq1\text{ or }d_{4j-i}=0\\
P/4&0,P/2&\nu_2(M_{ij})\geq1\text{ or }d_{4j-i}=0\\
P/4+c,\ 1\leq c<P/4&0,P/2+c,3P/4,3P/4+c
&\nu_2(M_{ij})\geq1\text{ or }d_{4j-i}=0\\
P/2&0,3P/4&\nu_2(M_{ij})\geq1\text{ or }d_{4j-i}=0.
\end{array}
\]
We give two representative substitutions, which also explain the first two
rows of the table.  Put $Q=P/4$.  First suppose $i<P$ and $4\nmid i$, and
write $e=\nu_2(i)\in\{0,1\}$.  Then $w_i=e$.  For $j=P,2P$, the two pairs
\[
\bigl(w_j,\nu_2(d_{4j-i})\bigr)
=\bigl(r+1,r-e+3\bigr)
\]
give valuation zero in \eqref{eq:073}.  For
$j=Q,P+Q,2P+Q$, the three pairs are
\[
(r-2,r-e),\qquad(r-1,r-e+1),\qquad(r-2,r-e),
\]
and again the valuation is zero.  These are exactly the residue classes
$0,Q$ modulo $P$ in the first row.

Next let $i=4c<P$, where $1\leq c<Q$.  Now
$w_i=2+\nu_2(c)$.  For the four residue classes
\[
j\equiv0, c, Q, Q+c\pmod P,
\]
substitution in \eqref{eq:073} and
\eqref{eq:074} gives valuation zero.  For example, the
three entries with $j\equiv c\pmod P$ use
\[
\begin{array}{c|ccc}
j&c&P+c&2P+c\\ \hline
w_j&\nu_2(c)&1+\nu_2(c)&\nu_2(c)\\
\nu_2(d_{4j-i})&0&1&0,
\end{array}
\]
whereas those with $j\equiv Q\pmod P$ use
\[
\begin{array}{c|ccc}
j&Q&P+Q&2P+Q\\ \hline
w_j&r-2&r-1&r-2\\
\nu_2(d_{4j-i})&r-2-\nu_2(c)&r-1-\nu_2(c)&r-2-\nu_2(c).
\end{array}
\]
The classes $j\equiv0$ and $j\equiv Q+c$ are identical one-line
substitutions.  The middle range $P\leq i\leq2P$ gives the third row in
the same way.  Finally, binomial symmetry and
$x^9\mathcal R(x^{-1})=\mathcal R(x)$ give
\[
w_{3P-i}=w_i,
\qquad d_{9P-\ell}=d_\ell,
\]
so that $M_{3P-i,3P-j}=M_{ij}$; this symmetry gives the last two rows from
the first two.  Thus the five rows are not independent computations.

The same substitutions also show that the right-hand sides of
\eqref{eq:073} and
\eqref{eq:074} are never negative.  Hence $M$ is
integral, and the table is exactly \eqref{eq:072} for $k=1$.

We record the inverse column supports before passing to higher powers.
Put $Q=P/4$ and fix $u=j\bmod P$.  Inverting the $k=1$ support table shows
that the indices $\ell$ for which $M_{\ell j}$ is odd are precisely
\begin{equation}\label{eq:cols}
\begin{array}{c|c}
u&\{\ell:1\leq\ell<3P,\ M_{\ell j}\equiv1\pmod2\}\\ \hline
0&\{1,2,\ldots,3P-1\}\\
Q&\{1,2,\ldots,P-1\}\\
2Q&\{P,P+1,\ldots,2P\}\\
3Q&\{2P+1,2P+2,\ldots,3P-1\}\\
0<u<Q\text{ or }Q<u<2Q&\{4(u\bmod Q)\}\\
2Q<u<3Q\text{ or }3Q<u<4Q
 &\{2P+4(u\bmod Q)\}.
\end{array}
\end{equation}
For example, if $0<u<Q$, the residue $u$ belongs only to
$E_u^{(Q)}$, and $\gamma_1(4u)=u$; the other nonboundary intervals are
identical.  The four boundary rows follow by reading the elements
$0,Q,2Q,3Q$ in the definition of the sets $E_c^{(Q)}$.

The induction step replaces a row support at scale $q$ by one at scale
$q/4$.  The inverse column table \eqref{eq:cols} reduces the matrix product
to the parity of the corresponding support intersections.

Assume \eqref{eq:072} for $k$, and suppose
$4^{k+1}\leq P$.  Write $q=4q'$.  Matrix multiplication is performed in
$\mathbb F_2$.  For $0\leq c\leq2q$, define, modulo $2$,
\[
N_c(j)=\sum_{\ell=1}^{3P-1}
\mathbf 1_{\ell\bmod4q\in E_c^{(q)}}
\mathbf 1_{j\bmod P\in E_{\gamma_1(\ell)}^{(P/4)}}.
\]
Here $\mathbf 1_{\mathcal E}$ is $1$ when the statement $\mathcal E$ is true and is
$0$ otherwise.
The first indicator is the support of a row of $M^k$, and the second is
the support of a row of $M$.  We calculate the parity of their
intersection using \eqref{eq:cols}.  Set
\[
S_c=\{\ell:1\leq\ell<3P,\ \ell\bmod4q\in E_c^{(q)}\}.
\]
Since $k\geq1$, the integer $4q$ divides $P$.  Each complete block of
$4q$ consecutive indices contains $|E_c^{(q)}|$ members of $S_c$.
The set $E_c^{(q)}$ contains $0$ and has even cardinality, either $2$ or
$4$.  For the four boundary columns $u=0,Q,2Q,3Q$ in
\eqref{eq:cols}, the intersection counts are, respectively,
\[
\frac{3P}{4q}|E_c^{(q)}|-1,\qquad
\frac{P}{4q}|E_c^{(q)}|-1,\qquad
\frac{P}{4q}|E_c^{(q)}|+1,\qquad
\frac{P}{4q}|E_c^{(q)}|-1.
\]
The signs account for the excluded endpoint $0$ or the two included
endpoints $P,2P$.  All four counts are odd.

For every other column, \eqref{eq:cols} gives a singleton,
$4(j\bmod Q)$ or $2P+4(j\bmod Q)$.  Because $q\mid Q$ and
$4q\mid2P$, this singleton lies in $S_c$ exactly when
$4j\bmod4q\in E_c^{(q)}$.  Equivalently,
\[
 j\bmod q\in
 \{a/4: a\in E_c^{(q)},\ 4\mid a\},
\]
where each element of $E_c^{(q)}$ is represented in
$\{0,1,\ldots,4q-1\}$.  This condition also holds for the four boundary
columns: there $j\bmod q=0$, and the set on the right contains $0$.
Thus it describes $N_c(j)$ for every column.

Recall that $q=4q'$.  Intersecting the five defining types of
$E_c^{(q)}$ with $4\mathbb Z$ and dividing by $4$ now gives the following
six cases; the middle column defines $\delta(c)$:
\[
\begin{array}{c|c|c}
c&\delta(c)&N_c(j)=1\text{ precisely when}\\ \hline
0\text{ or }0<c<q,\ 4\nmid c&0
&j\bmod4q'\in E_0^{(q')}\\
c=4d,\ 1\leq d<q'&d
&j\bmod4q'\in E_d^{(q')}\\
c=q&q'&j\bmod4q'\in E_{q'}^{(q')}\\
c=q+e,\ 1\leq e<q,\ 4\nmid e&2q'
&j\bmod4q'\in E_{2q'}^{(q')}\\
c=q+4d,\ 1\leq d<q'&q'+d
&j\bmod4q'\in E_{q'+d}^{(q')}\\
c=2q&2q'&j\bmod4q'\in E_{2q'}^{(q')}.
\end{array}
\]
For example, when $c=4d$ the residues $0,4d,q,q+4d$ become
$0,d,q',q'+d$, which form $E_d^{(q')}$.  When $0<c<q$ and
$4\nmid c$, only $0,q$ survive the intersection with $4\mathbb Z$,
giving $E_0^{(q')}$.  The middle and upper boundary sets become
$\{0,2q'\}$ and $\{0,3q'\}$, respectively.  These are exactly the
sets in the table.  The definition of
$\gamma_{k+1}$ says exactly that
$\delta(\gamma_k(i))=\gamma_{k+1}(i)$.  Since
\[
(M^{k+1})_{ij}=\sum_{\ell=1}^{3P-1}(M^k)_{i\ell}M_{\ell j},
\]
the table proves \eqref{eq:072} for $k+1$.
\end{proof}

\begin{lemma}\label{lem:20}
For every column $j$ of $M$ one has
\begin{align}
\sum_{\substack{1\leq\ell<3P\\
\ell\equiv0\text{ or }e\pmod4}}M_{\ell j}
&\equiv1\pmod2 &&(e=1,2,3),
\label{eq:075}\\
\sum_{\ell=1}^{3P-1}M_{\ell j}
&\equiv1\pmod2,
&
\sum_{\substack{1\leq\ell<3P\\2\mid\ell}}M_{\ell j}
&\equiv1\pmod2.
\label{eq:076}
\end{align}
\end{lemma}

\begin{proof}
By \eqref{eq:cols}, the odd entries in each column have one of the
six supports displayed there.
Every singleton in the last two rows is divisible by $4$.  In each of the
first four rows, the number of indices congruent to $0$ or to a fixed
$e\in\{1,2,3\}$ modulo $4$ is, respectively,
\[
3P/2-1,\qquad P/2-1,\qquad P/2+1,\qquad P/2-1,
\]
and the number of even indices is given by the same four odd integers.
All these numbers are odd because $4\mid P$.  The total number of indices in
the first four rows is, respectively,
$3P-1,P-1,P+1,P-1$, again always odd.  We have therefore proved, for every
column $j$, all three asserted identities.
\end{proof}

\begin{proof}[Proof of Lemma~\ref{lem:12}]
We combine the shrinking support pattern of
Lemma~\ref{lem:19} with the terminal column counts of
Lemma~\ref{lem:20}.

If $r=2a$ is even, take $k=a$ in
\eqref{eq:072}.  Then $q=1$: the lower, middle, and upper rows
are supported, respectively, on the residue pairs
\[
\{0,1\},\qquad\{0,2\},\qquad\{0,3\}\pmod4.
\]
Combining these three row supports with
\eqref{eq:075} shows that every entry of $M^{a+1}$ is
odd.

If $r=2a+1$ is odd, take $k=a$ in
\eqref{eq:072}, so that $q=2$.  One further multiplication gives
\[
(M^{a+1})_{ij}\equiv
\begin{cases}
1,&P\leq i\leq2P,\\
1,&i<P\text{ or }i>2P,\quad j\text{ even},\\
0,&i<P\text{ or }i>2P,\quad j\text{ odd}.
\end{cases}
\pmod2.
\]
This is obtained by inserting the five sets
$E_0^{(2)},\ldots,E_4^{(2)}$ in the matrix product; explicitly, the middle
set $E_2^{(2)}=\{0,4\}$ gives every column, whereas the two lower and two
upper sets give exactly the even columns.  Therefore the middle rows in
the last display, whose support contains all
$\ell$, and the remaining rows, whose support contains exactly the even
$\ell$, all become all-one rows after one more multiplication by $M$.
Thus $M^{a+2}\equiv\mathbf J\pmod2$.

Finally, once $M^k\equiv\mathbf J\pmod2$, the first identity in
\eqref{eq:076} gives
\[
M^{k+1}\equiv\mathbf JM\equiv\mathbf J\pmod2.
\]
Hence the congruence persists for every larger exponent, proving
\eqref{eq:045}.
\end{proof}

\subsubsection{The boundary sums}

\begin{proof}[Proof of Lemma~\ref{lem:13}]
Let
\[
\mathcal W(x)=\sum_{j=1}^{m-1}2^{w_j}x^j.
\]
By the Cauchy product, that is, the usual coefficient convolution,
$\mathcal S(n)$ is the coefficient of $x^n$ in
$\mathcal W(x)T(x)^m$.
By Lemma~\ref{lem:11}, modulo $2$,
\[
\mathcal W(x)
\equiv
\bigl(x+x^3+\cdots+x^{P-1}\bigr)(1+x^{2P}).
\]
Also
\[
T(x)^m\equiv(1+x)^{-m}=(1+x^P)^{-3}\pmod2.
\]
Since
\[
x+x^3+\cdots+x^{P-1}
=\frac{x(1+x^P)}{1+x^2}
\quad\text{in }\mathbb F_2[[x]],
\]
we obtain
\[
\mathcal W(x)T(x)^m
\equiv
\frac{x(1+x^P)(1+x^{2P})}
     {(1+x^2)(1+x^P)^3}
=\frac{x}{1+x^2}
\pmod2.
\]
The coefficient of every positive odd power is therefore $1$, proving the
first congruence.

For the second congruence, keep $\mathcal W$ modulo $4$.  Put
\[
O(x)=x+x^3+\cdots+x^{P-1},\qquad
E(x)=x^2+x^6+\cdots+x^{P-2}.
\]
The weight table gives
\[
\mathcal W(x)
\equiv
\bigl(O(x)+2E(x)\bigr)(1+x^{2P})
+2x^PO(x)
\pmod4.
\]
By \eqref{eq:053},
\[
T(x)^m\equiv(1+x)^{-m}\pmod4.
\]
Furthermore, \eqref{eq:063} gives
\[
(1+x)^{-m}\equiv(1+x^{P/2})^{-6}\pmod4,
\]
which contains only even powers.  Hence the terms involving $O(x)$ do not
contribute to an even coefficient, and
\[
\mathcal S(2q)
\equiv
2[x^{2q}]E(x)(1+x^{2P})(1+x)^{-m}
\pmod4.
\]
Modulo $2$,
\[
E(x)=\frac{x^2(1+x^P)}{1+x^4},
\qquad
(1+x)^{-m}=(1+x^P)^{-3}.
\]
Since $1+x^{2P}=(1+x^P)^2$ over $\mathbb F_2$, the expression after the
factor $2$ reduces to $x^2/(1+x^4)$.  Its coefficient at $x^{2q}$ is $1$
when $q$ is odd.  Thus $\mathcal S(2q)\equiv2\pmod4$.
\end{proof}

\subsection{Automaticity of the odd parts}
\label{sec:06}

We now combine Theorem~\ref{thm:val3} with the common
automaticity tools from Section~\ref{sec:01}.

\begin{retheorem}{thm:odd3}
Let $r\ge2$, put $m=3\cdot2^r$, and let $s\ge1$.  Then the sequence
\[
\bigl(\operatorname{odd}(t_m(n))\bmod2^s\bigr)_{n\ge0}
\]
is $2$-automatic.
\end{retheorem}

\begin{proof}[Proof of Theorem~\ref{thm:odd3}]
Retain the notation
\[
P=2^r,\qquad m=3P,\qquad
b(n)=\nu_2\binom{n+m-1}{m-1}
\]
from Section~\ref{sec:05}.  By
Theorem~\ref{thm:val3},
\begin{equation}\label{eq:077}
\nu_2(t_m(n))=b(n).
\end{equation}

We first dispose of the small-valuation range $b(n)\leq r+2$.  For each
$0\leq e\leq r+2$, the set
\[
E_e=\{n\geq0:\nu_2(t_m(n))=e\}
\]
is $2$-automatic, since membership is determined by
$t_m(n)\bmod2^{e+1}$ and Lemma~\ref{lem:04}.
The residue $t_m(n)\bmod2^{e+s}$ is also automatic.  On $E_e$, division
of this residue by $2^e$ is a well-defined map to $\mathbb Z/2^s\mathbb Z$.
Thus the odd part is automatic on every $E_e$, and
the sets $E_0,\ldots,E_{r+2}$ together with their complement form an
automatic partition.  Applying
Lemma~\ref{lem:06}(2), with an arbitrary constant
output on the complement, gives an automatic sequence agreeing with the
odd parts throughout the small-valuation range.

Suppose from now on that $b(n)\geq r+3$.  Recall the integer matrix
\[
M=\frac14DA^2D^{-1}
\]
from \eqref{eq:044}, and put
\[
Y_v=D\mathcal V_v.
\]
The exact identity \eqref{eq:050} says
\[
DA^{2k}=2^{2k}M^kD
\qquad(k\geq0).
\]

Write $n=Pq+\rho$, where $0\leq\rho<P$.  The construction in the proof of
Theorem~\ref{thm:val3} gives $h,\alpha,u$ such that
\[
n+h=2^\alpha u,\qquad u\text{ is odd},\qquad
b(n)=\alpha-w_h.
\]
If $\rho=0$, write $q=2a$ or $q=2a+1$.  The corresponding choices are
\[
(h,\alpha,u)=
\begin{cases}
(2P,r+1+\nu_2(a+1),\operatorname{odd}(a+1)),&q=2a,\\
(P,r+1+\nu_2(a+1),\operatorname{odd}(a+1)),&q=2a+1.
\end{cases}
\]
If $0<\rho<P$, write $q=4a+c$, where $1\leq c\leq3$. Then
\[
(h,\alpha,u)=
\bigl((4-c)P-\rho,r+2+\nu_2(a+1),\operatorname{odd}(a+1)\bigr).
\]
The case $0<\rho<P$ and $q\equiv0\pmod4$ cannot occur in the present
large-valuation range, by \eqref{eq:048}.

If $\alpha=2k$, then \eqref{eq:050}, applied to the
$h$-th coordinate, gives the exact identity
\begin{equation}\label{eq:079}
\frac{t_m(n)}{2^{b(n)}}=(M^kY_u)_h.
\end{equation}
If $\alpha=2k+1$, apply the same identity to $\mathcal V_{2u}$.  We obtain
\begin{equation}\label{eq:080}
\frac{t_m(n)}{2^{b(n)}}=\frac{(M^kY_{2u})_h}{2}.
\end{equation}
The right-hand side is an odd integer: this follows either from
\eqref{eq:077}, or directly from the final parity
calculation in the proof of Theorem~\ref{thm:val3}.

\smallskip
\noindent\emph{On each residue class.}
Fix $L=s+1$.  Lemma~\ref{lem:04} shows that
$(Y_v\bmod2^L)_{v\geq0}$ is $2$-automatic.  On each residue class displayed
above, the coordinate $h$ is fixed and
\[
\alpha=c_0+\nu_2(a+1),
\qquad
k=\left\lfloor\frac{c_0+\nu_2(a+1)}2\right\rfloor
\]
for a fixed integer $c_0$.  For this fixed $c_0$, the matrix sequence
\[
C_\ell=M^{\lfloor(c_0+\ell)/2\rfloor}\bmod2^L
\qquad(\ell\geq0)
\]
is eventually periodic, because the powers of $M$ modulo $2^L$ are
eventually periodic in a finite matrix ring.  Lemma~\ref{lem:05} shows that
$C_{\nu_2(a+1)}$, $Y_{\operatorname{odd}(a+1)}$, and
$Y_{2\operatorname{odd}(a+1)}$ are $2$-automatic.  Finite direct products
and matrix multiplication therefore show that both vector sequences
appearing on the right of \eqref{eq:079} and \eqref{eq:080} are
$2$-automatic modulo $2^L$.  The
choice between them is automatic as well, because
$c_0+\nu_2(a+1)\pmod2$ is obtained from an eventually periodic sequence by
Lemma~\ref{lem:05}.  Lemma
\ref{lem:06}(2) therefore joins the even- and
odd-$\alpha$ cases.  Extracting the fixed $h$-th coordinate and, in the odd
case, dividing an even residue modulo $2^{s+1}$ by $2$ are maps between
finite sets.  Hence each residue class carries an automatic sequence that
agrees with the odd parts at all large-valuation indices.

\smallskip
\noindent\emph{Combining the residue classes.}
The cases above are residue classes modulo $2P$ or $4P$.  After refining
the classes modulo $2P$ to classes modulo $4P$, Lemma
\ref{lem:06}(1) interleaves the finitely many automatic subsequences.  On
classes containing no large-valuation indices we use an arbitrary constant
output.  The large-valuation set is the complement of
$E_0\cup\cdots\cup E_{r+2}$ and is automatic.  Lemma~\ref{lem:06}(2)
then joins the small- and large-valuation ranges and proves the theorem.
\end{proof}

\section{The case \texorpdfstring{$m=6$}{m=6}}
\label{sec:04}

The same two-step matrix recurrence also treats $m=6$.  Here its first
and last coordinates acquire one additional factor of $2$, which gives
the correction in the valuation formula.

\begin{retheorem}{thm:six}
For every $n\ge0$,
\[
\nu_2(t_6(n))=\nu_2\binom{n+5}{5}
+\mathbf1_{n\equiv3\!\!\pmod4}.
\]
Moreover, for every $s\ge1$, the sequence
$\bigl(\operatorname{odd}(t_6(n))\bmod2^s\bigr)_{n\ge0}$ is $2$-automatic.
\end{retheorem}

\begin{proof}[Proof of Theorem~\ref{thm:six}]
We use $\mathcal V_n$ and $A$ from Lemma~\ref{lem:02}, with $m=6$.
Put
\[
(w_1,w_2,w_3,w_4,w_5)=(0,2,1,2,0),\qquad
D=\operatorname{diag}(1,4,2,4,1),
\]
and write
\[
Y_v=D\mathcal V_v=
\begin{pmatrix}
t_6(v-1)\\4t_6(v-2)\\2t_6(v-3)\\4t_6(v-4)\\t_6(v-5)
\end{pmatrix}.
\]
Direct multiplication gives $4KD=DA^2$, where
\[
K=
\begin{pmatrix}
4&-9&-18&1&0\\
9&29&-33&-15&1\\
-3&3&55&3&-3\\
1&-15&-33&29&9\\
0&1&-18&-9&4
\end{pmatrix}.
\]
Thus Lemma~\ref{lem:03} gives
\begin{equation}\label{eq:m6rec}
2^{2j}K^jY_v=Y_{2^{2j}v}\qquad(j,v\geq0).
\end{equation}
The matrix
\[
H=
\begin{pmatrix}
2&2&2&2&2\\
1&3&3&3&1\\
1&3&3&3&1\\
1&3&3&3&1\\
2&2&2&2&2
\end{pmatrix}
\]
satisfies $K^2\equiv H$ and $HK\equiv H\pmod4$.  Consequently,
\begin{equation}\label{eq:m6stab}
K^j\equiv H\pmod4\qquad(j\geq2).
\end{equation}

We need only one parity property of the vectors $Y_v$.  Modulo $2$,
\[
T(x)^6\equiv\frac1{(1+x^2)(1+x^4)}
\equiv\frac{1+x^2}{1+x^8}\pmod2.
\]
It follows that, for every positive odd integer $u$,
\[
\sum_{i=1}^5(Y_u)_i
\equiv t_6(u-1)+t_6(u-5)\equiv1\pmod2.
\]

Modulo $4$ is enough for the vectors $Y_u$; the vectors $Y_{2u}$
require one additional binary digit, so we use modulo $8$ below.
For any integral vector with odd coordinate sum, its product by $H$ has
odd middle three coordinates and first and last coordinates congruent
to $2$ modulo $4$.  Thus \eqref{eq:m6stab} determines the coordinate
valuations of $K^jY_u$ for $j\geq2$.  To treat $K^jY_{2u}$, note that
$DAD^{-1}$ is integral and that
\[
K^2DAD^{-1}\equiv KDAD^{-1}\pmod8.
\]
Let $e_h$ denote the row vector with its $h$-th entry equal to $1$
and all other entries equal to $0$.  Multiplying this congruence
repeatedly by $K$ and computing its first instance gives
\[
e_hK^jDAD^{-1}\equiv
\begin{cases}
4(1,1,1,1,1),&h=1,5,\\
2(3,1,1,1,3),&h=2,3,4
\end{cases}
\pmod8\qquad(j\geq1).
\]
Since $Y_{2u}=DAD^{-1}Y_u$ and the coordinate sum of $Y_u$ is odd,
these row congruences give valuation $2$ at the first and last
coordinates and valuation $1$ at the middle three coordinates.
Combining both calculations, we have
\begin{equation}\label{eq:m6coord}
\nu_2\bigl((K^jY_{2^\varepsilon u})_h\bigr)
=\varepsilon+\mathbf1_{h\in\{1,5\}}
\end{equation}
for $j\geq1$, $\varepsilon\in\{0,1\}$ and positive odd $u$, except
possibly when $j=1$, $\varepsilon=0$ and $h=1,5$.  Indeed, the
remaining case $j=1$, $\varepsilon=0$, $h=2,3,4$ follows because the
middle three rows of $K$ consist of odd integers.  The two excluded
coordinates will not be needed.

We now locate each coefficient in \eqref{eq:m6rec}.  Define
\[
b_6(n)=\nu_2\binom{n+5}{5}+\mathbf1_{n\equiv3\!\!\pmod4}.
\]
For $q\geq0$, put $a=\nu_2(q+1)$ and
$u=\operatorname{odd}(q+1)$.  The identity
$\binom{n+5}{5}=(n+1)\cdots(n+5)/120$ gives the following table:
\[
\begin{array}{c|c|c|c}
n&h&\alpha&b_6(n)\\ \hline
4q&4&2+a&a\\
4q+1&3&2+a&a+1\\
4q+2&2&2+a&a\\
8q+3&5&3+a&a+4\\
8q+7&1&3+a&a+4
\end{array}
\]
For example, in the row $n=8q+3$, the even factors in the numerator
have valuations $2,1,3+a$, while $\nu_2(120)=3$; the correction
adds $1$.  The other rows follow by the same factorization.  In each
row,
\[
n+h=2^\alpha u,\qquad
b_6(n)=\alpha-w_h+\mathbf1_{h\in\{1,5\}}.
\]
Write $\alpha=2j+\varepsilon$, where $\varepsilon\in\{0,1\}$.
Here $j\geq1$, and when $j=1$, $\varepsilon=0$, only $h=2,3,4$
occur.  Thus \eqref{eq:m6coord} applies in every case.  Taking the
$h$-th coordinate in \eqref{eq:m6rec} yields
\[
t_6(n)=2^{2j-w_h}(K^jY_{2^\varepsilon u})_h.
\]
Its valuation is $b_6(n)$, which proves the first assertion.  We have
also obtained the exact formula
\begin{equation}\label{eq:m6odd}
\operatorname{odd}(t_6(n))
=\frac{(K^jY_{2^\varepsilon u})_h}
{2^{\varepsilon+\mathbf1_{h\in\{1,5\}}}}.
\end{equation}
The divisor on the right is at most $4$.

Fix $s\geq1$.  By Lemma~\ref{lem:04} and closure under fixed shifts
and finite direct products, $(Y_v\bmod2^{s+2})_{v\geq0}$ is
$2$-automatic.  On each row of the table, $\alpha=c+\nu_2(q+1)$
with fixed $c\in\{2,3\}$.  The sequence
\[
\left(K^{\lfloor(c+\ell)/2\rfloor}\bmod2^{s+2},\,
(c+\ell)\bmod2\right)_{\ell\geq0}
\]
is eventually periodic: the powers of $K$ modulo $2^{s+2}$ take
values in a finite matrix ring.  Lemma~\ref{lem:05} therefore shows
that the matrix $K^j$ and the choice of $\varepsilon$, as functions
of $q$, are automatic.  The same lemma applies to
$Y_{\operatorname{odd}(q+1)}$ and $Y_{2\operatorname{odd}(q+1)}$.
Multiplication and selection of the $h$-th coordinate are maps between
finite sets.  Finally, division by the indicated factor of at most $4$
in \eqref{eq:m6odd} determines the result modulo $2^s$ from the
numerator modulo $2^{s+2}$.  Thus each row gives a $2$-automatic
sequence of odd parts.  Refining the first three rows into residue
classes modulo $8$ and applying Lemma~\ref{lem:06} proves the second
assertion.
\end{proof}

\section{Conjectures}
\label{sec:07}

The two conjectures in this section were formulated from the proved cases
and AI-assisted numerical experiments.  ChatGPT (OpenAI) was used in these
computations, and DeepSeek also contributed to the early numerical
exploration.  The computations served to guide the conjectures and are not
used in the proofs of the theorems.

\subsection{Automatic odd parts}
\label{sec:08}

The automaticity results above suggest the following extension to all
positive exponents.  We retain the convention $\operatorname{odd}(0)=0$
from the Introduction.

\begin{samepage}
\begin{conjecture}\label{conj:01}
For every $m\geq1$ and $s\geq1$, the sequence
\[
\bigl(\operatorname{odd}(t_m(n))\bmod2^s\bigr)_{n\geq0}
\]
is $2$-automatic.
\end{conjecture}
\end{samepage}

For $m=3$, some coefficients vanish. Their zero set is described in
\cite[Theorem~3.7]{GawronMiskaUlas} and is $2$-automatic by the recurrences in
\cite[Lemma~3.5 and Proposition~3.6]{GawronMiskaUlas}. Thus the zero
coefficients themselves cause no additional difficulty in Conjecture~\ref{conj:01}.

Lemma~\ref{lem:04} alone does not prove the
conjecture, because removing the full power of $2$ may require knowledge modulo
$2^{\nu_2(t_m(n))+s}$ and the required precision is not bounded in $n$.
Moreover, Appendix~\ref{app:02} shows that the analogous
statement fails for a general automatic coefficient sequence.  Thus the
conjecture depends on the special product structure of the Thue--Morse
series.

\subsection{A criterion for exact binomial valuation}

Let $m$ be a positive even integer, and write
\[
m=2^r u,
\qquad u\text{ odd},
\qquad r=\nu_2(m)\geq1.
\]
For a nonnegative integer $N$, let $\ell(N)$ be the maximum length of a block
of consecutive $1$'s in the binary expansion of $N$, with $\ell(0)=0$.

The valuation theorems suggest that the admissible exponents are
controlled by the binary expansion of the odd factor $u$.  More
precisely, we propose the following necessary and sufficient condition.

\begin{samepage}
\begin{conjecture}
\label{conj:02}
Let $m=2^r u$ be a positive even integer, where $u$ is odd and $r\geq1$.
Then
\[
\nu_2(t_m(n))=\nu_2\binom{n+m-1}{m-1}
\qquad\text{for every }n\geq0
\]
if and only if
\[
r\geq1+\ell\!\left(\frac{u-1}{2}\right).
\]
\end{conjecture}
\end{samepage}

For $u=1$, the criterion is $r\geq1$, in agreement with
Theorem~\ref{thm:pow2}.  For $u=3$, it is $r\geq2$; the positive cases
are Theorem~\ref{thm:val3}, and Theorem~\ref{thm:six} gives the failure
at $r=1$.  In general, the conjecture says that $2^r u$ has the exact
binomial valuation for every $n$ precisely when
\[
r\geq1+\ell\!\left(\frac{u-1}{2}\right).
\]

\section{Conclusion}

We proved exact $2$-adic valuation formulas and automaticity of the odd parts
for two infinite families of exponents and for $m=6$.  The two conjectures
ask whether odd-part automaticity holds for every exponent and which even
exponents satisfy the binomial valuation at every index.  The condition involving the longest block of consecutive $1$'s in the binary
expansion of $(u-1)/2$ may be approachable through the carry description in
\cite{GawronMiskaUlas}.  Appendix~\ref{app:02} shows that the first question
cannot follow from general closure properties of automatic or regular
sequences alone.

\appendix

\section{Odd parts need not be automatic}
\label{app:02}

The author anticipated that a counterexample should exist.  ChatGPT
(OpenAI) found the following example and made the principal contribution to
its proof.  The author checked the proof.

\begin{samepage}
\begin{theorem}\label{thm:09}
For $n\geq0$, let $a_n=-1$ when $n=2^j$ for some $j\geq0$,
and let $a_n=1$ otherwise.  Write
\[
F(x)=\sum_{n\ge0}a_nx^n,
\qquad F(x)^4=\sum_{n\ge0}c(n)x^n.
\]
Then $(a_n)$ is $2$-automatic, every $c(n)$ is nonzero, and
\[
\omega(n)=\operatorname{odd}(c(n))\pmod4
\]
is not $2$-automatic.
\end{theorem}
\end{samepage}

\begin{proof}
The powers of $2$ have binary expansions consisting of a single $1$
followed by zeros.  A finite automaton recognizes these positions, so
$(a_n)$ is $2$-automatic. Put $P(x)=\sum_{j\ge0}x^{2^j}$. Since
$F(x)=(1-x)^{-1}-2P(x)$, we have
\begin{equation}\label{eq:081}
F(x)^4=\frac1{(1-x)^4}-\frac{8P(x)}{(1-x)^3}
 +\frac{24P(x)^2}{(1-x)^2}-\frac{32P(x)^3}{1-x}+16P(x)^4.
\end{equation}
For $n\ge256$, let $L=\lfloor\log_2n\rfloor+1$. The two negative terms
in \eqref{eq:081} give the bound
\[
c(n)\ge\frac{n^3}{6}-4Ln(n+1)-32L^3.
\]
Indeed, $P$ has $L$ terms of degree at most $n$, each contributing at most
$\binom{n+1}{2}$ to $[x^n]P(x)/(1-x)^3$, and at most $L^3$ ordered triples
of these terms contribute to $[x^n]P(x)^3/(1-x)$.
For $L\ge9$, the inequality $2^{L-1}\ge28L$ follows by induction from
$L=9$. Thus $L\le n/28$, while $n+1\le9n/8$, and hence
\[
c(n)\ge n^3\left(\frac1{168}-\frac1{686}\right)>0.
\]
For $0\le n<256$, exact multiplication of
\[
\left(\sum_{j=0}^{255}x^j-2\sum_{j=0}^{7}x^{2^j}\right)^4
\]
gives
\[
\min_{0\leq n<256}|c(n)|=1.
\]
This finite calculation proves that the remaining coefficients are nonzero.

For $K\ge4$, set $M=2^K$ and
$Q(X)=6X^3-36X^2+27X+29$. Evaluating \eqref{eq:081} gives
\begin{equation}\label{eq:082}
3c(M+7)=\frac M2\bigl(M^2+(67-24K)M+144K^2-696K+602\bigr)-16Q(K).
\end{equation}
For the terms containing $P$ and $P^2$, sum the corresponding binomial
coefficients over powers of $2$.  The two remaining coefficients are
\[
[x^{M+7}]\frac{P(x)^3}{1-x}=K^3-3K+35,
\qquad [x^{M+7}]P(x)^4=24.
\]
Indeed, among the $K^3$ ordered triples of powers smaller than $M$,
the triples whose sum exceeds $M+7$ are the triple with all entries
$M/2$ and the $3(K-4)$ triples with two entries $M/2$ and the other
entry in $\{8,16,\ldots,M/4\}$.  The triples containing $M$ contribute
$3\cdot8=24$, because the ordered pairs of powers of $2$ with sum at
most $7$ number $8$.  This gives $K^3-1-3(K-4)+24=K^3-3K+35$.
Finally, $M+7$ has four nonzero binary digits, so a representation as a
sum of four powers of $2$ can have no carries.  Its summands are
$M,4,2,1$, giving $4!=24$.  These coefficient extractions yield
\eqref{eq:082}.

Since $Q(K)\equiv2\pmod3$, it is nonzero. If
$s=\nu_2(Q(K))\le K-7$, division of \eqref{eq:082} by $2^{s+4}$ shows that
$\nu_2(c(2^K+7))=s+4$ and
\begin{equation}\label{eq:083}
\omega(2^K+7)\equiv-3\operatorname{odd}(Q(K))\pmod4.
\end{equation}
Here the term containing $M/2$ is divisible by $4$ after division.

Suppose that $\omega$ is $2$-automatic. The binary expansion
$2^K+7=(10^{K-3}111)_2$ shows that $(\omega(2^K+7))_{K\ge4}$ is eventually
periodic: repeated input $0$ eventually cycles among finitely many states.
Let $p$ be a period valid for $K\ge N$, and put $t=\nu_2(p)$.
Choose $S\ge\max\{t+3,3\}$.

There is a unique odd root $a$ of $Q$ modulo $2^{S+2}$. Indeed,
$Q(X)\equiv X+1\pmod2$ and $Q'(x)$ is odd for every integer $x$.
For each root modulo $2^j$, exactly one of $a$ and $a+2^j$ is a root modulo
$2^{j+1}$, by Taylor expansion. This proves the assertion by induction.
For $y=1,3$, choose an integer $K_y\ge\max\{N,S+7\}$ satisfying
\[
K_y\equiv a+2^Sy\pmod{2^{S+2}}.
\]
Taylor expansion now gives
\[
Q(K_y)\equiv2^SyQ'(a)\pmod{2^{S+2}}.
\]
Thus $\nu_2(Q(K_y))=S$, and the two residues
$\operatorname{odd}(Q(K_y))\equiv yQ'(a)\pmod4$ are distinct.

On the other hand,
\[
Q(a+p)-Q(a)=p\bigl(Q'(a)+18(a-2)p+6p^2\bigr),
\]
whose factor in parentheses is odd. Consequently
$\nu_2(Q(K_y+p))=t$. The two arguments $K_y+p$ are congruent modulo
$2^{t+2}$, so their values under $Q$, divided by $2^t$, are congruent
modulo $4$. Therefore $\operatorname{odd}(Q(K_y+p))\pmod4$ is independent
of $y$.
The precision condition in \eqref{eq:083} holds at both $K_y$ and $K_y+p$.
The assumed period would therefore identify the two distinct residues at
$K_y$ with the same residue at $K_y+p$, a contradiction.
\end{proof}

\section{Initial coefficient data for the case $r=2$}
\label{app:r2data}

For the initial case in Lemma~\ref{lem:18}, direct evaluation of the finite
coefficient sum gives
\[
\begin{aligned}
(d_{4q})_{q=0}^{9}
  &=(1,-231,8052,22276,-95634,-95634,22276,8052,-231,1),\\
(d_{4q+1})_{q=0}^{8}
  &=(-12,1056,-528,-69024,-42504,89056,25200,-3168,-76),\\
(d_{4q+2})_{q=0}^{8}
  &=(54,-880,-21912,20976,134596,20976,-21912,-880,54),\\
(d_{4q+3})_{q=0}^{8}
  &=(-76,-3168,25200,89056,-42504,-69024,-528,1056,-12).
\end{aligned}
\]
These values are included only to make the finite initial calculation
checkable; the proof of Lemma~\ref{lem:18} uses their $2$-adic valuations.

\section*{Declaration of generative AI and AI-assisted technologies
in the manuscript preparation process}

The author used ChatGPT (OpenAI) for mathematical exploration, numerical
computation, and assistance with the writing and revision of this paper.
DeepSeek was also used in the early stages of numerical exploration.  The
author provided the core ideas for the main results.  For
Appendix~\ref{app:02}, the author anticipated the existence of a
counterexample; ChatGPT found the counterexample and made the principal
contribution to the proof of the theorem in that appendix.  The conjectures
in Section~\ref{sec:07} were informed by AI-assisted numerical
computations.  These exploratory computations are not used in the proofs of
the theorems.  The author checked all proofs, including the proof in
Appendix~\ref{app:02}, reviewed and revised the AI-assisted content, and
takes full responsibility for the content of this paper.

\end{document}